\pdfoutput=1
\documentclass[12pt]{amsart}
\usepackage{oldgerm}
\usepackage{latexsym}
\usepackage{amsmath}
\usepackage{amscd}
\usepackage{amssymb} % \mathfrak , \mathbb
\usepackage{amsmath}
\usepackage{amsthm}
\usepackage{latexsym}
\usepackage{arydshln}
\makeatletter
\@namedef{subjclassname@2020}{%
  \textup{2020} Mathematics Subject Classification}
\makeatother
\subjclass[2020]{11F46, 11F67}
\begin{document}
\title [Boundedness of $L$-values ]{Boundedness of denominators of special values of the $L$-functions for  modular forms
}  

\author{Hidenori KATSURADA}
\address{Muroran Institute of Technology
27-1 Mizumoto, Muroran, 050-8585, Japan}
\email{hidenori@mmm.muroran-it.ac.jp}
%\footnotetext[1]{Partially supported by GSPS KAKENHI }
\date{\today}

\keywords{Standard $L$-function, Siegel modular form }
\thanks{The author is partially supported by JSPS KAKENHI Grant Number (B) No.16H03919.}

\maketitle
%Greek letters
\newcommand{\alp}{\alpha}
\newcommand{\bet}{\beta}
\newcommand{\gam}{\gamma}
\newcommand{\del}{\delta}
\newcommand{\eps}{\epsilon}
\newcommand{\zet}{\zeta}
\newcommand{\tht}{\theta}
\newcommand{\iot}{\iota}
\newcommand{\kap}{\kappa}
\newcommand{\lam}{\lambda}
\newcommand{\sig}{\sigma}
\newcommand{\ups}{\upsilon}
\newcommand{\ome}{\omega}
\newcommand{\vep}{\varepsilon}
\newcommand{\vth}{\vartheta}
\newcommand{\vpi}{\varpi}
\newcommand{\vrh}{\varrho}
\newcommand{\vsi}{\varsigma}
\newcommand{\vph}{\varphi}
\newcommand{\Gam}{\Gamma}
\newcommand{\Del}{\Delta}
\newcommand{\Tht}{\Theta}
\newcommand{\Lam}{\Lambda}
\newcommand{\Sig}{\Sigma}
\newcommand{\Ups}{\Upsilon}
\newcommand{\Ome}{\Omega}

%fraktur letters

\newcommand{\frka}{{\mathfrak a}}    \newcommand{\frkA}{{\mathfrak A}}
\newcommand{\frkb}{{\mathfrak b}}    \newcommand{\frkB}{{\mathfrak B}}
\newcommand{\frkc}{{\mathfrak c}}    \newcommand{\frkC}{{\mathfrak C}}
\newcommand{\frkd}{{\mathfrak d}}    \newcommand{\frkD}{{\mathfrak D}}
\newcommand{\frke}{{\mathfrak e}}    \newcommand{\frkE}{{\mathfrak E}}
\newcommand{\frkf}{{\mathfrak f}}    \newcommand{\frkF}{{\mathfrak F}}
\newcommand{\frkg}{{\mathfrak g}}    \newcommand{\frkG}{{\mathfrak G}}
\newcommand{\frkh}{{\mathfrak h}}    \newcommand{\frkH}{{\mathfrak H}}
\newcommand{\frki}{{\mathfrak i}}    \newcommand{\frkI}{{\mathfrak I}}
\newcommand{\frkj}{{\mathfrak j}}    \newcommand{\frkJ}{{\mathfrak J}}
\newcommand{\frkk}{{\mathfrak k}}    \newcommand{\frkK}{{\mathfrak K}}
\newcommand{\frkl}{{\mathfrak l}}    \newcommand{\frkL}{{\mathfrak L}}
\newcommand{\frkm}{{\mathfrak m}}    \newcommand{\frkM}{{\mathfrak M}}
\newcommand{\frkn}{{\mathfrak n}}    \newcommand{\frkN}{{\mathfrak N}}
\newcommand{\frko}{{\mathfrak o}}    \newcommand{\frkO}{{\mathfrak O}}
\newcommand{\frkp}{{\mathfrak p}}    \newcommand{\frkP}{{\mathfrak P}}
\newcommand{\frkq}{{\mathfrak q}}    \newcommand{\frkQ}{{\mathfrak Q}}
\newcommand{\frkr}{{\mathfrak r}}    \newcommand{\frkR}{{\mathfrak R}}
\newcommand{\frks}{{\mathfrak s}}    \newcommand{\frkS}{{\mathfrak S}}
\newcommand{\frkt}{{\mathfrak t}}    \newcommand{\frkT}{{\mathfrak T}}
\newcommand{\frku}{{\mathfrak u}}    \newcommand{\frkU}{{\mathfrak U}}
\newcommand{\frkv}{{\mathfrak v}}    \newcommand{\frkV}{{\mathfrak V}}
\newcommand{\frkw}{{\mathfrak w}}    \newcommand{\frkW}{{\mathfrak W}}
\newcommand{\frkx}{{\mathfrak x}}    \newcommand{\frkX}{{\mathfrak X}}
\newcommand{\frky}{{\mathfrak y}}    \newcommand{\frkY}{{\mathfrak Y}}
\newcommand{\frkz}{{\mathfrak z}}    \newcommand{\frkZ}{{\mathfrak Z}}

%math boldface latters

\newcommand{\bfa}{{\mathbf a}}    \newcommand{\bfA}{{\mathbf A}}
\newcommand{\bfb}{{\mathbf b}}    \newcommand{\bfB}{{\mathbf B}}
\newcommand{\bfc}{{\mathbf c}}    \newcommand{\bfC}{{\mathbf C}}
\newcommand{\bfd}{{\mathbf d}}    \newcommand{\bfD}{{\mathbf D}}
\newcommand{\bfe}{{\mathbf e}}    \newcommand{\bfE}{{\mathbf E}}
\newcommand{\bff}{{\mathbf f}}    \newcommand{\bfF}{{\mathbf F}}
\newcommand{\bfg}{{\mathbf g}}    \newcommand{\bfG}{{\mathbf G}}
\newcommand{\bfh}{{\mathbf h}}    \newcommand{\bfH}{{\mathbf H}}
\newcommand{\bfi}{{\mathbf i}}    \newcommand{\bfI}{{\mathbf I}}
\newcommand{\bfj}{{\mathbf j}}    \newcommand{\bfJ}{{\mathbf J}}
\newcommand{\bfk}{{\mathbf k}}    \newcommand{\bfK}{{\mathbf K}}
\newcommand{\bfl}{{\mathbf l}}    \newcommand{\bfL}{{\mathbf L}}
\newcommand{\bfm}{{\mathbf m}}    \newcommand{\bfM}{{\mathbf M}}
\newcommand{\bfn}{{\mathbf n}}    \newcommand{\bfN}{{\mathbf N}}
\newcommand{\bfo}{{\mathbf o}}    \newcommand{\bfO}{{\mathbf O}}
\newcommand{\bfp}{{\mathbf p}}    \newcommand{\bfP}{{\mathbf P}}
\newcommand{\bfq}{{\mathbf q}}    \newcommand{\bfQ}{{\mathbf Q}}
\newcommand{\bfr}{{\mathbf r}}    \newcommand{\bfR}{{\mathbf R}}
\newcommand{\bfs}{{\mathbf s}}    \newcommand{\bfS}{{\mathbf S}}
\newcommand{\bft}{{\mathbf t}}    \newcommand{\bfT}{{\mathbf T}}
\newcommand{\bfu}{{\mathbf u}}    \newcommand{\bfU}{{\mathbf U}}
\newcommand{\bfv}{{\mathbf v}}    \newcommand{\bfV}{{\mathbf V}}
\newcommand{\bfw}{{\mathbf w}}    \newcommand{\bfW}{{\mathbf W}}
\newcommand{\bfx}{{\mathbf x}}    \newcommand{\bfX}{{\mathbf X}}
\newcommand{\bfy}{{\mathbf y}}    \newcommand{\bfY}{{\mathbf Y}}
\newcommand{\bfz}{{\mathbf z}}    \newcommand{\bfZ}{{\mathbf Z}}

%caligraphic letters

%\newcommand{\cal}{\fam2}
\newcommand{\cala}{{\mathcal A}}
\newcommand{\calb}{{\mathcal B}}
\newcommand{\calc}{{\mathcal C}}
\newcommand{\cald}{{\mathcal D}}
\newcommand{\cale}{{\mathcal E}}
\newcommand{\calf}{{\mathcal F}}
\newcommand{\calg}{{\mathcal G}}
\newcommand{\calh}{{\mathcal H}}
\newcommand{\cali}{{\mathcal I}}
\newcommand{\calj}{{\mathcal J}}
\newcommand{\calk}{{\mathcal K}}
\newcommand{\call}{{\mathcal L}}
\newcommand{\calm}{{\mathcal M}}
\newcommand{\caln}{{\mathcal N}}
\newcommand{\calo}{{\mathcal O}}
\newcommand{\calp}{{\mathcal P}}
\newcommand{\calq}{{\mathcal Q}}
\newcommand{\calr}{{\mathcal R}}
\newcommand{\cals}{{\mathcal S}}
\newcommand{\calt}{{\mathcal T}}
\newcommand{\calu}{{\mathcal U}}
\newcommand{\calv}{{\mathcal V}}
\newcommand{\calw}{{\mathcal W}}
\newcommand{\calx}{{\mathcal X}}
\newcommand{\caly}{{\mathcal Y}}
\newcommand{\calz}{{\mathcal Z}}

%math script

\newcommand{\scra}{{\mathscr A}}
\newcommand{\scrb}{{\mathscr B}}
\newcommand{\scrc}{{\mathscr C}}
\newcommand{\scrd}{{\mathscr D}}
\newcommand{\scre}{{\mathscr E}}
\newcommand{\scrf}{{\mathscr F}}
\newcommand{\scrg}{{\mathscr G}}
\newcommand{\scrh}{{\mathscr H}}
\newcommand{\scri}{{\mathscr I}}
\newcommand{\scrj}{{\mathscr J}}
\newcommand{\scrk}{{\mathscr K}}
\newcommand{\scrl}{{\mathscr L}}
\newcommand{\scrm}{{\mathscr M}}
\newcommand{\scrn}{{\mathscr N}}
\newcommand{\scro}{{\mathscr O}}
\newcommand{\scrp}{{\mathscr P}}
\newcommand{\scrq}{{\mathscr Q}}
\newcommand{\scrr}{{\mathscr R}}
\newcommand{\scrs}{{\mathscr S}}
\newcommand{\scrt}{{\mathscr T}}
\newcommand{\scru}{{\mathscr U}}
\newcommand{\scrv}{{\mathscr V}}
\newcommand{\scrw}{{\mathscr W}}
\newcommand{\scrx}{{\mathscr X}}
\newcommand{\scry}{{\mathscr Y}}
\newcommand{\scrz}{{\mathscr Z}}

%math Bbb

\newcommand{\AAA}{{\mathbb A}} %not \AA
\newcommand{\BB}{{\mathbb B}}
\newcommand{\CC}{{\mathbb C}}
\newcommand{\DD}{{\mathbb D}}
\newcommand{\EE}{{\mathbb E}}
\newcommand{\FF}{{\mathbb F}}
\newcommand{\GG}{{\mathbb G}}
\newcommand{\HH}{{\mathbb H}}
\newcommand{\II}{{\mathbb I}}
\newcommand{\JJ}{{\mathbb J}}
\newcommand{\KK}{{\mathbb K}}
\newcommand{\LL}{{\mathbb L}}
\newcommand{\MM}{{\mathbb M}}
\newcommand{\NN}{{\mathbb N}}
\newcommand{\OO}{{\mathbb O}}
\newcommand{\PP}{{\mathbb P}}
\newcommand{\QQ}{{\mathbb Q}}
\newcommand{\RR}{{\mathbb R}}
\newcommand{\SSS}{{\mathbb S}} %not \SS
\newcommand{\TT}{{\mathbb T}}
\newcommand{\UU}{{\mathbb U}}
\newcommand{\VV}{{\mathbb V}}
\newcommand{\WW}{{\mathbb W}}
\newcommand{\XX}{{\mathbb X}}
\newcommand{\YY}{{\mathbb Y}}
\newcommand{\ZZ}{{\mathbb Z}}

%typewriter

\newcommand{\tta}{\hbox{\tt a}}    \newcommand{\ttA}{\hbox{\tt A}}
\newcommand{\ttb}{\hbox{\tt b}}    \newcommand{\ttB}{\hbox{\tt B}}
\newcommand{\ttc}{\hbox{\tt c}}    \newcommand{\ttC}{\hbox{\tt C}}
\newcommand{\ttd}{\hbox{\tt d}}    \newcommand{\ttD}{\hbox{\tt D}}
\newcommand{\tte}{\hbox{\tt e}}    \newcommand{\ttE}{\hbox{\tt E}}
\newcommand{\ttf}{\hbox{\tt f}}    \newcommand{\ttF}{\hbox{\tt F}}
\newcommand{\ttg}{\hbox{\tt g}}    \newcommand{\ttG}{\hbox{\tt G}}
\newcommand{\tth}{\hbox{\tt h}}    \newcommand{\ttH}{\hbox{\tt H}}
\newcommand{\tti}{\hbox{\tt i}}    \newcommand{\ttI}{\hbox{\tt I}}
\newcommand{\ttj}{\hbox{\tt j}}    \newcommand{\ttJ}{\hbox{\tt J}}
\newcommand{\ttk}{\hbox{\tt k}}    \newcommand{\ttK}{\hbox{\tt K}}
\newcommand{\ttl}{\hbox{\tt l}}    \newcommand{\ttL}{\hbox{\tt L}}
\newcommand{\ttm}{\hbox{\tt m}}    \newcommand{\ttM}{\hbox{\tt M}}
\newcommand{\ttn}{\hbox{\tt n}}    \newcommand{\ttN}{\hbox{\tt N}}
\newcommand{\tto}{\hbox{\tt o}}    \newcommand{\ttO}{\hbox{\tt O}}
\newcommand{\ttp}{\hbox{\tt p}}    \newcommand{\ttP}{\hbox{\tt P}}
\newcommand{\ttq}{\hbox{\tt q}}    \newcommand{\ttQ}{\hbox{\tt Q}}
\newcommand{\ttr}{\hbox{\tt r}}    \newcommand{\ttR}{\hbox{\tt R}}
\newcommand{\tts}{\hbox{\tt s}}    \newcommand{\ttS}{\hbox{\tt S}}
\newcommand{\ttt}{\hbox{\tt t}}    \newcommand{\ttT}{\hbox{\tt T}}
\newcommand{\ttu}{\hbox{\tt u}}    \newcommand{\ttU}{\hbox{\tt U}}
\newcommand{\ttv}{\hbox{\tt v}}    \newcommand{\ttV}{\hbox{\tt V}}
\newcommand{\ttw}{\hbox{\tt w}}    \newcommand{\ttW}{\hbox{\tt W}}
\newcommand{\ttx}{\hbox{\tt x}}    \newcommand{\ttX}{\hbox{\tt X}}
\newcommand{\tty}{\hbox{\tt y}}    \newcommand{\ttY}{\hbox{\tt Y}}
\newcommand{\ttz}{\hbox{\tt z}}    \newcommand{\ttZ}{\hbox{\tt Z}}

\newcommand{\phm}{\phantom}
\newcommand{\ds}{\displaystyle }
\newcommand{\smallstrut}{\vphantom{\vrule height 3pt }}
\def\bdm #1#2#3#4{\left(
\begin{array} {c|c}{\ds{#1}}
 & {\ds{#2}} \\ \hline
{\ds{#3}\vphantom{\ds{#3}^1}} &  {\ds{#4}}
\end{array}
\right)}
\newcommand{\wtd}{\widetilde }
\newcommand{\bsl}{\backslash }
\newcommand{\GL}{{\mathrm{GL}}}
\newcommand{\SL}{{\mathrm{SL}}}
\newcommand{\GSp}{{\mathrm{GSp}}}
\newcommand{\PGSp}{{\mathrm{PGSp}}}
\newcommand{\SP}{{\mathrm{Sp}}}
\newcommand{\SO}{{\mathrm{SO}}}
\newcommand{\SU}{{\mathrm{SU}}}
\newcommand{\Ind}{\mathrm{Ind}}
\newcommand{\Hom}{{\mathrm{Hom}}}
\newcommand{\Ad}{{\mathrm{Ad}}}
\newcommand{\Sym}{{\mathrm{Sym}}}
\newcommand{\Mat}{\mathrm{M}}
\newcommand{\sgn}{\mathrm{sgn}}
\newcommand{\trs}{\,^t\!}
\newcommand{\iu}{\sqrt{-1}}
\newcommand{\oo}{\hbox{\bf 0}}
\newcommand{\ono}{\hbox{\bf 1}}
\newcommand{\smallcirc}{\lower .3em \hbox{\rm\char'27}\!}
\newcommand{\bAf}{\bA_{\hbox{\eightrm f}}}
\newcommand{\thalf}{{\textstyle{\frac12}}}
\newcommand{\shp}{\hbox{\rm\char'43}}
\newcommand{\Gal}{\operatorname{Gal}}
\newcommand{\St}{\mathrm{St}}
\newcommand{\bdel}{{\boldsymbol{\delta}}}
\newcommand{\bchi}{{\boldsymbol{\chi}}}
\newcommand{\bgam}{{\boldsymbol{\gamma}}}
\newcommand{\bome}{{\boldsymbol{\omega}}}
\newcommand{\bpsi}{{\boldsymbol{\psi}}}
\newcommand{\GK}{\mathrm{GK}}
\newcommand{\EGK}{\mathrm{EGK}}
\newcommand{\ord}{\mathrm{ord}}
\newcommand{\diag}{\mathrm{diag}}
\newcommand{\ua}{{\underline{a}}}
\newcommand{\ZZn}{\ZZ_{\geq 0}^n}

\newtheorem{theorem}{Theorem}[section]
\newtheorem{lemma}[theorem]{Lemma}
\newtheorem{proposition}[theorem]{Proposition}
\newtheorem{corollary}[theorem]{Corollary}
\newtheorem{conjecture}[theorem]{Conjecture}
\newtheorem{definition}[theorem]{Definition}
\newtheorem{remark}[theorem]{{\bf Remark}}

%\newtheorem*{corollary}{\bf {Corollary}}
%%%%%%%%%%%%%%%%title%%%%%%%%%%%%%%%%%%%
%

\def\mattwono(#1;#2;#3;#4){\begin{array}{cc}
                               #1  & #2 \\
                               #3  & #4
                                      \end{array}}

\def\mattwo(#1;#2;#3;#4){\left(\begin{matrix}
                               #1 & #2 \\
                               #3  & #4
                                      \end{matrix}\right)}
 \def\smallmattwo(#1;#2;#3;#4){\left(\begin{smallmatrix}
                               #1 & #2 \\
                               #3  & #4
                                      \end{smallmatrix}\right)}                                     
                                      
 \def\matthree(#1;#2;#3;#4;#5;#6;#7;#8;#9){\left(\begin{matrix}
                               #1 & #2  & #3\\
                               #4  & #5 & #6\\
                               #7  & #8 &#9 
                                      \end{matrix}\right)}                                     
                                      
\def\mattwo(#1;#2;#3;#4){\left(\begin{matrix}
                               #1 & #2 \\
                               #3  & #4
                                      \end{matrix}\right)}  

\def\rowthree(#1;#2;#3){\begin{matrix}
                               #1   \\
                               #2  \\
                               #3
                                      \end{matrix}}  
\def\columnthree(#1;#2;#3){\begin{matrix}
                               #1   &   #2  &  #3
                                      \end{matrix}}  
                                      
\def\rowfive(#1;#2;#3;#4;#5){\begin{array}{lllll}
                               #1   \\
                               #2  \\
                               #3 \\
                               #4 \\
                               #5                              
                                      \end{array}} 

\def\columnfive(#1;#2;#3;#4;#5){\begin{array}{lllll}
                               #1   &   #2  &  #3 & #4 & #5
                                \end{array}}

\def\mattwothree(#1;#2;#3;#4;#5;#6){\begin{matrix}
                               #1 & #2  & #3  \\
                               #4 & #5  & #6
                                      \end{matrix}}  
\def\matthreetwo(#1;#2;#3;#4;#5;#6){\begin{array}{lc}
                               #1  & #2  \\
                               #3  & #4 \\
                               #5  & #6
                                      \end{array}}  
\def\columnthree(#1;#2;#3){\begin{matrix}
                               #1 & #2 & #3  
                                  \end{matrix}}  
\def\rowthree(#1;#2;#3){\begin{matrix}
                               #1 \\
                                #2 \\
                                #3  
                                  \end{matrix}}  
\def\smallddots{\mathinner
{\mskip1mu\raise3pt\vbox{\kern7pt\hbox{.}}
\mskip1mu\raise0pt\hbox{.}
\mskip1mu\raise-3pt\hbox{.}\mskip1mu}}

\begin{abstract}
For a cuspidal Hecke eigenform $F$ for $Sp_n(\ZZ)$ and a Dirichlet character $\chi$ let 
$L(s,F,\chi,\St)$ be the standard $L$-function of $F$ twisted by $\chi$.
In \cite{Boecherer17}, B\"ocherer showed the boundedness of denominators 
of the algebraic part of $L(m,F,\chi,\St)$ at a critical point $m$ when $\chi$ varies.  In this paper, we give  a refined version of his  result.  We  also prove a similar result for the products of Hecke $L$-functions of primitive forms for $SL_2(\ZZ)$. 
\end{abstract}
\section{Introduction}
Let $\varGamma^{(n)}=Sp_n(\ZZ)$ be the Siegel modular group of genus $n$. 
For a cuspidal Hecke eigenform $F$ for $\varGamma^{(n)}$ and a Dirichlet character $\chi$ let 
$L(s,F,\chi,\St)$ be the standard $L$-function of $F$ twisted by $\chi$.
In \cite{Boecherer17}, B\"ocherer showed the boundedness of denominators 
of the algebraic part of $L(m,F,\chi,\St)$ at a critical point $m$ when $\chi$ varies (cf. Remark \ref{rem.main-result}).
 To prove this, B\"ocherer used congruence of Fourier coefficients of modular forms. In this paper, we give a refined version of the above result without using congruence. We state our main results more precisely. Let $M_k(\varGamma^{(n)})$ be the space of modular forms of weight $k$ for $\varGamma^{(n)}$, and $S_k(\varGamma^{(n)})$ its subspace consisting of cusp forms. We suppose that $k \ge n+1$. 
Let $F_1,\ldots, F_e$ be a basis of the space  $M_k(\varGamma^{(n)})$ consisting of Hecke eigenforms such that $F_1=F$. 
Let $L_{n,k}$ be the composite field of $\QQ(F_1), \cdots,\QQ(F_{e-1})$ and $\QQ(F_e)$. 
Let $\widetilde \frkE_F'$ be the ideal of  $L_{n,k}$ generated by all $\prod_{i=2}^{e}(\lambda_F(T_{i-1})-\lambda_{F_i}(T_{i-1}))$'s \ $(T_1,\ldots,T_{e-1} \in {\bf L}_n')$ and put $\widetilde \frkE_F =\widetilde \frkE_F' \cap \QQ(F)$,
where ${\bf L}_n'$ is the Hecke algebra for the Hecke pair $(  GSp_n^+(\QQ) \cap M_{2n}(\ZZ), \varGamma^{(n)})$.  Then, by Theorem \ref{th.mult-one}, 
$\widetilde \frkE_F'$ is a non-zero ideal, and therefore  $\widetilde \frkE_F$ is a non-zero ideal of $\QQ(F)$.
 Let $\frkI(l,F,\chi)$ be a certain fractional ideal of $\QQ(F,\chi)$ associated with the value $L(l,F,\chi,\St)$ as defined in Section 2, where $\QQ(F,\chi)$ is the field generated over the Hecke field $\QQ(F)$ of $F$ by all the values of $\chi$. 
 Then we prove that we have 
\[\frkI(m,F,\chi) \subset  \langle  (C_{n,k}\widetilde \frkE_F)^{-1} \rangle_{\frkO_{\QQ(F,\chi)}[N^{-1}]}\]
for any positive integer $m \le k-n$ and primitive character $\chi$ mod $N$  satisfying a certain condition, where $C_{n,k}$ is a positive integer depending  only on $k$ and $n$. (For a precise statement, see Theorem \ref{th.main-result}). By this we easily see the following result (cf. Corollary \ref{cor.main-result1-2}):

\bigskip

{\it Let $\calp_F$ be the set  of prime ideals $\frkp$ of $\QQ(F)$ such that 
\[\ord_{\frkp}(N_{\QQ(F,\chi)/\QQ(F)}(\frkI(m,F,\chi))) <0 \]
for some positive integer $m \le k-n$ and primitive character $\chi$ with conductor  not divisible by $\frkp$ satisfying the above condition.
Then $\calp_F$ is a finite set. Moreover, there exists a positive integer $r=r_{n,k}$ depending only on $n$ and $k$ such that  we have 
\[\ord_{\frkq}(\frkI(m,F,\chi)) \ge -r[\QQ(F,\chi):\QQ(F)]\]
for any prime ideal $\frkq$ of $\QQ(F,\chi)$  lying above a prime ideal in $\calp_F$ and  positive integer $m \le k-n$ and  primitive  character $\chi$ with conductor not divisible by $\frkq$ satisfying the above condition.}

\bigskip

We have also similar results for the products of Hecke $L$ functions of primitive forms for $SL_2(\ZZ)$. 

The author thanks Shih-Yu Chen, Tobias Keller, Takashi Ichikawa, and Masataka Chida  for valuable discussions. He also thanks the referee for many useful comments.

{\bf Notation}
 We denote by $\ZZ_{> 0}$ and $\ZZ_{\ge 0}$ the set of positive integers and the set of non-negative integers, respectively.                                            

For a commutative ring $R$, let $M_{mn}(R)$ denote  the set of $m \times n$ matrices with entries                                     
 in $R$, and especially write $M_n(R)=M_{nn}(R)$. 
We often identify an element $a$ of $R$ and the matrix $(a)$ of size 1 whose component is $a$. If $m$ or $n$ is 0, we understand an element of $M_{mn}(R)$ is the {\it empty matrix} and denote it by $\emptyset$. Let $GL_n(R)$ be the group consisting of all invertible elements of $M_n(R)$, and ${\rm Sym}_n(R)$ the set of symmetric matrices of size $n$ with entries in $R$.  
%For a  semigroup  $S$ we put $S^{\Box}=\{s^2 \ | \ s \in S \}$.%
Let $K$ be a field of characteristic $0$, and $R$ its subring. We say that an element  $A $ of $\mathrm{Sym}_n(R)$ is non-degenerate if the determinant  $\det A$ of $A$ is non-zero. For a subset $S$ of $\mathrm{Sym}_n(R)$, we denote by
$S^{{\rm{nd}}}$ the subset of $S$ consisting of non-degenerate matrices.  For a subset $S$ of $\mathrm{Sym}_n(\RR)$ we denote by $S_{\ge 0}$  (resp. $S_{>0}$) the subset of $S$ consisiting of semi-positive definite (resp. positive definite) matrices. We say that an element  $A=(a_{ij})$ of $\mathrm{Sym}_n(K)$ is half-integral if $a_{ii} \ (i=1,...,n)$ and $2a_{ij} \ (1 \le i \not= j \le n)$ belong to $R$. We denote by $\calh_n(R)$ the set of half-integral matrices of size $n$ over $R$. 
We note that $\calh_n(R)=\mathrm {Sym}_n(R)$ if $R$ contains the inverse of $2$. 
 For an $(m,n)$ matrix $X$ and an $(m,m)$ matrix $A$, we write $A[X] ={}^tXAX$, where $^t X$ denotes the transpose of $X$.  
Let $G$ be a subgroup of $GL_n(R)$. Then we say that two elements $B$ and $B'$ in $\mathrm{Sym}_n(R)$  are $G$-equivalent if there is an element $g$ of $G$ such that $B'=B[g]$. 
For two square matrices $X$ and $Y$ we write $X \bot Y =\mattwo(X;O;O;Y)$. We often write $x \bot Y$ instead of $(x) \bot Y$ if $(x)$ is  a matrix of size 1. 
We denote by $1_m$ the unit matrix of size $m$ and by $O_{m,n}$ the zero matrix of type $(m,n)$. We sometimes abbreviate $O_{m,n}$ as $O$ if there is no fear of  confusion.

Let $\frkb$ be a subset of $K$. We then denote by $\langle  \frkb \rangle_R$ the $R$-sub-module of $K$ generated by $\frkb$.  For a non-zero integer $M$, we  put  
\[R[M^{-1}]=\{a M^{-s} \ | \ a \in R, \ s \in \ZZ_{\ge 0}\}\] 
Let $K$ be an algebraic number filed, and $\frkO=\frkO_K$  the ring of integers in $K.$ For  a prime ideal ${\frkp}$ of ${\frkO},$ we denote by ${\frkO}_{({\frkp})}$ the localization of ${\frkO}$ at ${\frkp}$ in $K.$ Let $\frkA$ be a fractional ideal in $K.$ If $\frkA={\frkp}^e \frkB$ with a fractional ideal $\frkB$ of $K$ such that ${\frkO}_{({\frkp})}\frkB={\frkO}_{({\frkp})}$  we write ${\rm ord}_{{\frkp}}(\frkA)=e.$
We make the convention that $\ord_\frkp(\frkA)=\infty$ if $\frkA=\{0\}$. We simply write ${\rm ord}_{{\frkp}}(c)= {\rm ord}_{{\frkp}}((c))$ for $c \in K.$ For an ideal $\frkI$ of $K$, let $\frkI^{-1}$ the inverse ideal of $\frkI$.

For a complex number $x$ put ${\bf e}(x)=\exp(2\pi \sqrt{-1}x)$.

\section{Main result}
For a subring $K$ of $\RR$ put  
\[GSp_n^+(K)=\{\gamma \in GL_{2n}(K)  \ | \  J_n[\gamma]= \kappa(\gamma)J_n \
{\rm with \ some} \ \kappa(\gamma) > 0 \},\]
 and
\[Sp_n(K)=\{\gamma \in GSp_n^+(K)  \ | \  J_n[\gamma]=J_n \}, \]
 where $J_n=\mattwo(O_n;-1_n;1_n;O_n)$. 
In particular, put $\varGamma^{(n)}=Sp_n(\ZZ)$ as in  Introduction.
We sometimes write an element $\gamma$ of $GSp_n^+(K)$ as 
$\gamma=\mattwo(A;B;C;D)$ with $A,B,C,D \in M_n(K).$ We define subgroups $\varGamma^{(n)}(N)$ and $\varGamma_0^{(n)}(N)$ of $\varGamma^{(n)}$ as
$$\varGamma^{(n)}(N)=\{\gamma \in \varGamma^{(n)} \ | \ \gamma \equiv 1_{2n}  \text{ mod } N \},$$
and
$$\varGamma_0^{(n)}(N)=\{\mattwo(A;B;C;D) \in \varGamma^{(n)} \ | \  C \equiv O_n \text{ mod }  N \}.$$
Let ${\bf H}_n$ be Siegel's upper half space of degree $n$. We write 
$\gamma(Z)=(AZ+B)(CZ+D)^{-1}$ and $j(\gamma,Z)=\det (CZ+D)$ for $\gamma=\mattwo(A;B;C;D) \in GSp_n^+({\RR})$ and $Z \in {\bf H}_n$. We write $F|_k\gamma(Z)=(\det \gamma)^{k/2} j(\gamma,Z)^{-k}f(\gamma (Z))$ for $\gamma \in GSp_n^+({\RR})$ and a $C^{\infty}$-function  $F$ on ${\bf H}_n.$ We simply write $F|\gamma$ for $F|_k\gamma $ if there is no confusion. 
We say that a subgroup $\varGamma$ of $\varGamma^{(n)}$ is a congruence subgroup if $\varGamma$ contains $\varGamma^{(n)}(N)$ with some $N$. We also say that a character $\eta$ of a congruence subgroup $\varGamma$ is a congruence character if its kernel is a congruence subgroup. For a positive integer $k$, a congruence subgroup $\Gamma$ and its congruence character $\eta$, we denote by $M_k(\Gamma,\eta)$  (resp. $M_k^{\infty}(\varGamma,\eta)$) the space of holomorphic (resp. $C^{\infty}$-) modular forms 
 of weight $k$ and character $\eta$ for $\Gamma$. We denote by $S_k(\varGamma,\eta)$ the subspace of $M_k(\varGamma,\eta)$ consisting of cusp forms. If $\eta$ is the trivial character, we abbreviate $M_k(\varGamma,\eta)$ and $S_k(\varGamma,\eta)$ as $M_k(\varGamma)$ and $S_k(\varGamma)$, respectively. 
Let $dv$ denote the invariant volume element on ${\bf H}_n$ defined by 
\[dv=\det ({\rm Im}(Z))^{-n-1} \wedge_{1 \le j \le l \le n}(dx_{jl} \wedge dy_{jl}).\]
 Here for $Z \in {\bf H}_n$ we write $Z= (x_{jl}) + \sqrt{-1} (y_{jl})$ with real matrices $(x_{jl})$ and $(y_{jl}).$
For two elements  $F$ and $G$ of $M_k^{\infty}(\varGamma,\eta)$, we define 
the Petersson scalar product $\langle F,G\rangle  _{\varGamma}$ of $F$ and $G$ by
$$\langle F,G\rangle _{\varGamma}=\int_{\varGamma \backslash {\bf H}_n} F(Z)\overline {G(Z)} \det ({\rm Im}(Z))^k dv,$$
provided the integral converges. For $i=1,2$, let $\varGamma_i$ be a congruence subgroup with a congruence character $\eta_i$.
Then there exists a congruence subgroup $\varGamma$ contained in $\varGamma_1 \cap \varGamma_2$ and its congruence character $\eta$ such that $\eta_1|\Gamma=\eta_2|\Gamma=\eta$. Then we have $M_k^\infty(\varGamma,\eta) \supset M_k^\infty(\varGamma_i,\eta_i)$. For elements $F_1$ and $F_2$ of $M_k^\infty(\varGamma_,\eta_1)$ and $M_k^\infty(\varGamma_2,\eta_2)$, respectively, the value  $[\varGamma^{(n)}:\varGamma]^{-1}\langle F_1, \ F_2 \rangle_{\varGamma}$  does not depend on the choice of $\varGamma$. 
We denote it by $\langle F_1, \ F_2 \rangle$. 

Let $F$ be an element of $M_k(\varGamma,\eta)$. Then, $F$ has the following Fourier expansion:
\[F(Z)=\sum_{A \in \calh_n(\ZZ)_{\ge 0}} c_F\bigl({A \over N}\bigr){\bf e}\bigl(\mathrm{tr}({AZ \over N})\bigr)\]
with some positive integer $N$, where $\mathrm{tr}$ denotes the trace of a matrix. 
For a subset  $S$ of $\CC$, we denote by $M_k(\varGamma,\eta)(S)$ the set of elements $F$ of $M_k(\varGamma,\eta)$  such that $c_F({A \over N}) \in S$ for all $A \in \calh_n(\ZZ)_{\ge 0}$, and put $S_k(\varGamma,\eta)(S)=M_k(\varGamma,\eta)(S) \cap S_k(\varGamma,\eta)$. If $R$ is a commutative ring, and $S$ is an $R$ module, then $M_k(\varGamma,\eta)(S)$ and $S_k(\varGamma,\eta)(S)$ are $R$-modules.

 For a Dirichlet character $\phi$ modulo $N$, let  $\widetilde \phi$  denote   the character of $\varGamma_0^{(n)}(N)$  defined by $\varGamma_0^{(n)}(N) \ni \left(\begin{matrix} A & B \\C & D
\end{matrix}\right) \mapsto \phi(\det D)$, and we write $M_k(\varGamma_0^{(n)}(N),\phi)$ for $M_k(\varGamma_0^{(n)}(N),\widetilde \phi)$, and so on.

We denote by ${\bf L}_n={\bf L}_{\QQ}(GSp_n^+(\QQ), \varGamma^{(n)})$ be the Hecke ring over $\QQ$ associated with the Hecke pair $(GSp_n^+(\QQ),\varGamma^{(n)}))$, and by
${\bf L}_n'={\bf L}_{\ZZ}(GSp_n^+(\QQ) \cap M_{2n}(\ZZ),\varGamma^{(n)})$ be the Hecke ring over $\ZZ$ associated with the Hecke pair $(GSp_n^+(\QQ) \cap M_{2n}(\ZZ),\varGamma^{(n)})$. For a Hecke eigenform $F$, we denote by $\QQ(F)$ the field generated over $\QQ$ by the eigenvalues of all Hecke operators $T \in {\bf L}_n$ with respect to $F$, and call it the Hecke field of $F$. For Dirichlet characters $\chi_1,\ldots,\chi_r$, we denote by $\QQ(\chi_1,\ldots,\chi_r)$ the field generated over $\QQ$ by all the values of $\chi_1,\ldots,\chi_r$, and by $\QQ(F,\chi_1,\ldots,\chi_r)$ the composite field of $\QQ(F)$ and $\QQ(\chi_1,\ldots,\chi_r)$.
For a Hecke eigenform $F$ in $S_k(\varGamma_0^{(n)}(N))$  and a Dirichlet character $\chi$ let
$L(s,F,\mathrm{St},\chi)$ be the standard $L$ function of $F$ twisted by $\chi$. For a Dirichlet character $\chi$,  we put  $\delta_{\chi}=0$ or $1$ according as $\chi(-1)=1$ or $\chi(-1)=-1$.
Assume that $\chi$ is primitive, and for any positive integer $m \le k-n$ such that $m-n \equiv \delta_{\chi} \text{ mod } 2$ 
define $\Lambda(m,F,\chi,\St)$ as 
\[\Lambda(m,F,\chi,\St)={\chi(-1)^n
\Gamma(m)\prod_{i=1}^n \Gamma(2k-n-i)L(m,F, {{\mathrm St}},\chi) \over \langle F, F \rangle \pi^{-n(n+1)/2+nk+(n+1)m}\sqrt{-1}^{m+n}\tau(\chi)^{n+1}}.\]
$\tau(\chi)$ is the Gauss sum of $\chi$.
 For a Dirichlet character $\chi$ let $m_{\chi}$ be the conductor of $\chi$.
The following proposition is essentially due to [\cite{Boecherer-Schmidt00}, Appendix, Theorem].
\begin{proposition} \label{prop.algebraicity-standard-L}
Let $F$ be a Hecke eigenform in $S_k(\varGamma^{(n)})(\QQ(F))$.
Let $m$ be a positive integer not greater than $k-n$ and $\chi$  a  primitive character $\chi$ satisfying the following condition:

{\rm (C)}  $m-n \equiv \delta_{\chi} \text{ mod } 2$, and  $m >1$ if $n>1, \ n \equiv 1 \ \text{ mod }  4$ and $\chi^2$ is trivial.

Then $\Lambda(m,F,\chi,\St)$ belongs to $\QQ(F,\chi)$.
\end{proposition}

 Let $\calv$ be a subspace of $M_k(\varGamma^{(n)})$. We say that a multiplicity one holds for $\calv$ if 
any Hecke eigenform in $\calv$ is uniquely determined up to constant multiple  by its Hecke eigenvalues.
\begin{theorem}
\label{th.mult-one}
Suppose that $k \ge n+1$. Then
a  multiplicity one theorem holds for $S_k(\varGamma^{(n)})$.
\end{theorem}
\begin{proof}
This is essentially due to Chenevier-Lannes [\cite{Chenevier-Lannes19}, Corollary 8.5.4].  It was proved under a more stronger assumption without using [\cite{Chenevier-Lannes19}, Conjecture 8.4.22]. As is written in the postface in that book, this conjecture has been proved \cite{Arancibia-Moeglin-Renard18}, and the same proof is available at least even when $k \ge n+1$.
\end{proof}
Let $F$ be a Hecke eigenform in $S_k(\varGamma^{(n)})$ with $k \ge n+1$. Then by Theorem \ref{th.mult-one}, we have $cF \in S_k(\varGamma^{(n)})(\QQ(F))$ with some $c \in \CC$.
Hence for $A,B \in \calh_n(\ZZ)_{>0}$ and an integer $l$ satisfying $(C)$, the value $c_F(A)\overline{c_F(B)}{\Lambda}(l,F,\mathrm{St},\chi)$ belongs to $\QQ(F)$ and does not depend on the choice of $c$.  For $A$ and $B$ and an integer $l$ put  
 $$I_{A,B}(l,F,\chi)=c_{F}(A)\overline{c_{F}(B)} {\Lambda}(l,F,\chi,\St).$$
  Let ${\mathfrak  I}(l,F,\chi)$ be  the ${\frkO}_{\QQ(F)}$-module  generated by all $I_{A,B}(l,F,\chi)$'s. Then  ${\mathfrak  I}_F(l,F,\chi)$ becomes a  fractional ideal in $\QQ(F, \chi).$  We note that it  is uniquely determined by $l$ and the system of eigenvalues of $F$. 
Let $F_1,\ldots,F_d$ be a basis of $S_k(\varGamma^{(n)})$ consisting of Hecke eigenforms such that $F_1=F$.
Let $K_{n,k}$ be the composite filed $\QQ(F_1)\cdots \QQ(F_d)$ of $\QQ(F_1),\ldots, \QQ(F_d)$. 
We denote by $\widetilde \frkD_{F}'$ the ideal of  $K_{n,k}$ generated by all
$\prod_{i=2}^d (\lambda_F(T_{i-1})-\lambda_{F_i}(T_{i-1}))$'s $( T_1,\cdots, T_{d-1} \in  {\bf L}_n'$), and 
put $\widetilde \frkD_F=\frkD_F' \cap \QQ(F)$. We make the convention that
$\widetilde \frkD_F'=\frkO_{K_{n,k}}$ if $d=1$. Moreover, let $\widetilde \frkE_F$
be the ideal of $\QQ(F)$ defined in Section 1. 
Then our first main result is as follows.
\begin{theorem}
\label{th.main-result} Let $F$ be a Hecke eigenform in $S_k(\varGamma^{(n)})$. 
Then we have 
\[  \frkI(m,F,\chi) \subset  \langle (2^{\alpha(n,k)} A_{n,k} \widetilde \frkE_{F})^{-1}\rangle_{\frkO_{\QQ(F,\chi)}[N^{-1}]}\]
 for any positive  integer $m \le k-n$  and   primitive character  $\chi$ mod $N$ satisfying the condition $\mathrm {(C)}$, where $\alpha(n,k)$ is a non-negative integer depending only on $k$ and $n$, and
$A_{n,k}=\mathrm{LCM}_{n+1 \le m \le k} \{\prod_{i=1}^n (2l-2i)(2l-2i+1)!)\}$. In particular if $m \le k-n-1$, then
\[  \frkI(m,F,\chi) \subset  \langle  (2^{\alpha(n,k)} A_{n,k} \widetilde \frkD_{F})^{-1} \rangle_{\frkO_{\QQ(F,\chi)[N^{-1}]}}\]
\end{theorem}
We will prove the above theorem in Section 5. 
\begin{corollary} \label{cor.main-result1-2} 
Let $F$ be a Hecke eigenform in $S_k(\varGamma^{(n)})$. 
Let $\calp_F$ be the set  of prime ideals $\frkp$ of $\QQ(F)$ such that 
\[\ord_{\frkp}(N_{\QQ(F,\chi)/\QQ(F)}(\frkI(m,F,\chi))) <0 \]
for some positive integer $m \le k-n$ and primitive character $\chi$ with conductor not divisible by $\frkp$ satisfying $\mathrm{(C)}$.
Then $\calp_F$ is a finite set. Moreover, there exists a positive integer $r$ such that  we have 
\[\ord_{\frkq}(\frkI(m,F,\chi)) \ge -r[\QQ(F,\chi):\QQ(F)]\]
for any prime ideal $\frkq$ of $\QQ(F,\chi)$  lying above a prime ideal in $\calp_F$ and  integer $l$ and  primitive  character $\chi$ with conductor not divisible by $\frkq$ satisfying the condition {\rm (C)}. 
\end{corollary}
\begin{proof}
By Theorem \ref{th.main-result}, we have $\frkp | 2^{\alpha(n,k)}A_{n,k}\widetilde \frkE_F$ if $\frkp \in \calp_F$. This proves the first assertion.
Let $2^{\alpha(n,k)}A_{n,k}\widetilde \frkE_F=\frkp_1^{e_1}\cdots \frkp_s^{e_s}$ be the prime factorization of $2^{\alpha(n,k)}A_{n,k}\widetilde \frkE_F$, where 
$\frkp_1,\ldots,\frkp_s$ are distinct prime ideals and $e_1,\ldots,e_s$ are positive integers.
We note that for any prime ideal $\frkp$ of $\QQ(F)$ and prime ideal $\frkq$ of $\QQ(F,\chi)$ lying above $\frkp $ we have $\ord_{\frkq}(\frkp) \le [\QQ(F,\chi):\QQ(F)]$. Hence $r=\max \{e_i \}_{1 \le i \le s}$ satisfies the required condition in the second assertion.
\end{proof}
\begin{remark}
\label{rem.main-result}
{\rm (1)} Let 
\[ \Lambda(F,m,\chi)={
\Gamma(m)\prod_{i=1}^n \Gamma(2k-n-i)L(m,F, {{\mathrm St}},\chi) \over \langle F, F \rangle \pi^{-n(n+1)/2+nk+(n+1)m}.}\]
Then, if $m$ and $\chi$ satisfy the condition {\rm (C)}, $\Lambda(F,m,\chi)$ belongs to $\QQ(F,\chi,\zeta_N)$, where  $\QQ(F,\chi,\zeta_N)$ is the field generated over the Hecke field $\QQ(F)$ of $F$ by all the values of $\chi$ and the primitive $N$-th root $\zeta_N$ of unity. In {\rm [\cite{Boecherer17}, Theorem]}, a similar result has been proved for $\Lambda(F,m,\chi)$. Our $L$-value belongs to $\QQ(F,\chi)$, which is included in $\QQ(F,\chi,\zeta_N)$. Therefore, our result can be regarded  as a refinement of B\"ocherer's.  

\noindent 
{\rm (2)} B\"ocherer \cite{Boecherer17}  excluded the case $m =k-n$. However, we can include this case. We also note that we can get a sharper result if we restrict ourselves to the case $m <k-n$ as stated in the above theorem.

\noindent
{\rm (3)} In \cite{Boecherer17}, the main result was  formulated without assuming multiplicity one theorem. 
 However, such a formulation is now unnecessary.
\end{remark}

\section {Pullback of Siegel Eisenstein series}
To prove our main result, first we express a certain modular form  as a linear combination of Hecke eigenforms (cf. Theorem \ref{th.pullback-Eisenstein}).
We have carried out it in [\cite{Katsurada15}, Appendix], and here 
we  treat it in a more general setting. We also correct some inaccuracies in [\cite{Katsurada15}, Appendix] (cf. Remark \ref{rem.corrections}).
For a non-negative integer $m$, put
\[\Gamma_m(s)=\pi^{m(m-1)/4}\prod_{i=1}^m \Gamma(s-{i-1 \over 2}).\]
 For a Dirichlet character $\chi$ we denote by $L(s,\chi)$ the Dirichlet $L$-function associated to $\chi,$ and put
\begin{align*}
&\call_{n}(s,\chi)= \Gamma_{n}(s) \pi^{-ns} L(s,\chi) \prod_{i=1}^{[n/2]} L(2s-2i,\chi^2)\\
& \times \begin{cases} \pi^{n/2-s}\Gamma(s-n/2) & \text{ if } \text{ if } n \text{ is even}\\
1 & n \text{ is odd.} \end{cases}
\end{align*}
 Let $n,l$ and $N$ be positive integers.  For a Dirichlet character $\phi$ modulo $N$ such that $\phi(-1)=(-1)^l,$ we define the Eisenstein series $E_{n,l}^*(Z;N,\phi,s)$ by
$$E_{n,l}^*(Z;N, \phi,s)=\bigl(\det {\rm Im}(Z)\bigr)^s  \call_n(l+2s,\phi)$$
$$ \times \sum_{\gamma \in T^{(n)}(N)_\infty \backslash T^{(n)}(N)}  \phi^*(\gamma) j(\gamma,Z)^{-l}|j(\gamma,Z)|^{-2s},$$
where 
\[T^{(n)}(N)=\Bigl\{\begin{pmatrix} A & B \\ C & D \end{pmatrix} \in \varGamma^{(n)} \ | \ A \equiv O_n \text{ mod } N \Bigr\},\]
\[T^{(n)}(N)_{\infty}=\Bigl\{\begin{pmatrix} A & B \\ C & D \end{pmatrix} \in \varGamma^{(n)} \ | \ B \equiv O_n \text{ mod } N, C=O_n \Bigr\},\]
and $\phi^*(\gamma)=\phi(\det C)$ for $\gamma =\begin{pmatrix} A & B \\ C & D \end{pmatrix} \in T^{(n)}(N)$.
 Then $E_{n,l}^*(Z;N,\phi,s)$ converges absolutely as a function of $s$ if the real part of $s$ is large enough.
Moreover, it has a meromorphic continuation to the whole $s$-plane, and  it  belongs to  
$M_l^{\infty}(\varGamma_0^{(n)}(N),\phi)$. Moreover it is holomorphic and finite at $s=0$, which will be denoted by $E_{n,l}^*(Z;N,\phi)$.
In particular, if $E_{n,l}^*(Z;N,\phi)$ belongs to $M_l(\varGamma_0^{(n)}(N),\phi),$ it has the following Fourier expansion:
$$E_{n,l}^*(Z;N,\phi)=\sum_{A \in {\calh}_n(\ZZ)_{\ge 0}} c_{n,l}\left(A,N,\phi) {\bf e}({\rm tr}(AZ)\right).$$
To see the Fourier coefficient of $E_{n,l}^*(Z;N,\phi)$, we define a polynomial attached to local Siegel series.
For a prime number $p$ let $\QQ_p$ be the field of $p$-adic numbers, and $\ZZ_p$ the ring of $p$-adic integers. For an element $B \in \calh_n(\ZZ_p)$,
we define the Siegel series $b_p(B,s)$  as 
\[b_p(B,s)=\sum_{R \in \mathrm{Sym}_n(\QQ_p)/\mathrm{Sym}_n(\ZZ_p)} 
{\bf e}_p(\mathrm{tr}(BR))\nu(R)^{-s},\]
where ${\bf e}_p$ is the additive character of $\ZZ_p$ such that ${\bf e}_p(m)={\bf e}(m)$ for $m \in \ZZ[p^{-1}]$, and
$\nu_p(R)=[R\ZZ_p^n+\ZZ_p^n:\ZZ_p^n]$.
We define $\chi_p(a)$ for $a \in {\QQ}^{\times}_p $ as follows:
 $$\chi_p(a):=  
 \left\{\begin{array}{cl} 
  +1            & {\rm if } \ {\QQ}_p(\sqrt {a})={\QQ}_p ,\\
  -1            & {\rm if}  \ {\QQ}_p(\sqrt {a})/{\QQ}_p \ {\rm is \ quadratic \ unramified},\\
  0             & {\rm if}  \ {\QQ}_p(\sqrt {a})/{\QQ}_p \ {\rm is \ quadratic \ ramified}. 
\end{array}\right.$$
For an element  $B \in \calh_n(\ZZ_p)^{\rm nd}$ with $n$ even,  we define $\xi_p(B)$ by 
$$\xi_p(B):=\chi_p((-1)^{n/2}\det B).$$ 
%Let $\frkD_B$ be the discriminant of% $\QQ(\sqrt{(-1)^{n/2}\det B})/\QQ)$, and \[\frke_B=(\ord(\det (2B))-\ord(\frkD_B)+1-\xi_p(B)^2)/2.\]%
%We denote the Clifford invariant of $B$ by $\eta_B$ (cf.% %\cite{Ikeda-Katsurada18}. \changed{reference?)} %
For a nondegenerate half-integral matrix $B$ of size $n$ over ${\ZZ}_p$ define a polynomial $\gamma_p(B,X)$ in $X$ by
$$\gamma_p(B,X):=
\left\{
\begin{array}{ll}
(1-X)\prod_{i=1}^{n/2}(1-p^{2i}X^2)(1-p^{n/2}\xi_p(B)X)^{-1} & \ {\rm if} \ n \ {\rm is \ even}, \\
(1-X)\prod_{i=1}^{(n-1)/2}(1-p^{2i}X^2) & \ {\rm if} \ n \ {\rm is \ odd}.
\end{array}
\right.$$
Then it is well known that  there exists a unique polynomial $F_p(B,X)$ in $X$ over   ${\ZZ}$  with constant term $1$ such that 
$$b_p(B,s) =\gamma_p(B,p^{-s})F_p(B,p^{-s})$$ (e.g. \cite{Katsurada99}). More precisely, we have the following proposition.
\begin{proposition} \label{prop.precise-Kitaoka-polynomial}
Let $B \in \calh_m(\ZZ_p)^{\mathrm{nd}}$. Then there exists a polynomial $H_p(B,x)$ in $X$ over $\ZZ$ such that 
\[F_p(B,X)=H_p(B,p^{[(m+1)/2]}X).\]
\end{proposition}
\begin{proof}
The assertion follows from \cite{Kitaoka84}, Theorem 2.
\end{proof}
                             
For  $B \in \calh_m(\ZZ)_{>0}$ with $m$ even, 
 let ${\mathfrak d}_B$ be the discriminant of  ${\QQ}(\sqrt{(-1)^{m/2}\det B})/{\QQ}$,  and $\chi_B=({\frac{{\mathfrak d}_B}{*}})$  the Kronecker character corresponding to ${\QQ}(\sqrt{(-1)^{m/2}\det B})/{\QQ}$. We note that we have $\chi_B(p)=\xi_p(B)$ for any prime $p.$                                    
 We also note that 
\[(-1)^{m/2}\det (2B)=\frkd_B \frkf_B^2\]
with $\frkf_B \in \ZZ_{>0}$.
We define a polynomial $F_p^*(T,X)$ for any $T \in \calh_n(\ZZ_p)$ which is not-necessarily non-degenerate as follows:
For an element $T \in \calh_n(\ZZ_p)$ of rank $r \ge 1,$ there exists an element 
$\widetilde T \in \calh_r(\ZZ_p)^{{\rm nd}}$ such that $T \sim_{\ZZ_p} \widetilde T \bot O_{n-r}.$ We note that  
$F_p(\widetilde T,X)$ does not depend on the choice of $\widetilde T.$ Then we put $F_p^\ast(T,X)=F_p(\widetilde T,X).$  
For an element $T \in \calh_n(\ZZ)_{\ge 0}$ of rank $r \ge 1,$ there exists an element $\widetilde T \in 
\calh_r(\ZZ)_{>0}$ such that $T \sim_{\ZZ} \widetilde T \bot O_{n-r}.$
Then $\chi_{\widetilde T}$ does not depend on the choice of $\widetilde T$. We write $\chi_T^{\ast}=\chi_{\widetilde T}$ if $r$ is even. 
 For a non-negative integer $m$ and a primitive character $\phi$ let $B_{m,\phi}$ be the $m$-th generalized Bernoulli number for $\phi$. In the case $\phi$ is the principal character, we write $B_m=B_{m,\phi}$, which is the $m$-th Bernoulli number.  For a Dirichlet character $\phi$ we denote by $\phi_0$ the primitive character associated with $\phi$.
\begin{proposition} \label{prop.FC-Siegel-Eisenstein}
Let $n$ and $l$ be positive integers such that $l \ge n+1$, and $\phi$ a primitive character $\text{ mod } N$. 
 Then $E_{2n,l}^*(Z;N,\phi)$ is holomorphic and belongs to $M_l(\varGamma_0^{(2n)}(N),\phi)$ except the following case:

$l=n+1 \equiv 2 \text{ mod } 4$ and $\phi^2={\bf 1}_N$.

In the case that $E_{2n,l}^*(Z;N,\phi)$ is holomorphic we have the following assertion:
\begin{itemize}
\item[(1)] Suppose that  $N=1$ and $\phi$ is the principal character ${\bf 1}$, 
Then 
for $B \in \calh_{2n}(\ZZ)_{\ge 0}$ of rank $m,$ we have
\begin{align*}
c_{2n,l}(B,1,{\bf 1})
&=(-1)^{l/2+n(n+1)/2}2^{l-1+[(m+1)/2]}\prod_{p \mid  \det (2\widetilde B)} F_p^\ast(B,p^{l-m-1})\\
&\times \left\{
\begin{cases}
\prod_{i=m/2+1}^{n} \zeta(1+2i-2l) L(1+m/2-l,\chi_B^\ast)
 & \ {\rm if} \ m \ {\rm is \ even  }, \\
\prod_{i=(m+1)/2}^{n} \zeta(1+2i-2l) & \ {\rm if} \ m \ {\rm is \ odd}.
\end{cases}
\right.
\nonumber
\end{align*}
Here we make the convention that $F_p^*(B,p^{l-m-1})=1$ and $\L(1+m/2-l,\chi^*_B)=\zeta(1-l)$ if $m=0$.
\item[(2)] Suppose that $N>1$. Then, $c_{2n,l}(B,N,\phi)=0$ if $B \in \calh_{2n}(\ZZ)_{\ge 0}$ is not positive definite.
Moreover, for any $B \in \calh_{2n}(\ZZ)_{>0}$ we have
\begin{align*}
&c_{2n,l}(B,N,\phi)=(-1)^{nl+(l-n-\delta_{(\phi \chi_B)_0})/2}2^{n+l-1}\sqrt{-1}^{-\delta_{(\phi \chi_B)_0}}|\frkd_B|^{l-n-1/2}\\
&\times m_{(\phi\chi_B)_0}^{n-l}\tau((\phi\chi_B)_0)  \prod_p F_p(B,p^{l-2n-1}\bar \phi(p))
 L(1-l+n,\overline{(\phi \chi_B)_0})\\
& \times \prod_{p | N|\frkd_B|} (1-p^{n-l}(\phi\chi_B)_0).
\end{align*}
\end{itemize}
\end{proposition}
\begin{proof}  (1) The assertion  follows from [\cite{Katsurada10}, Theorem 2.3] remarking that
\[\call_{2n}(l,\chi)=\zeta(1-l)\prod_{i=1}^n \zeta(1-2l+2i)(-1)^{(n(n+1)+l)/2}2^{l-1}.\]
(2) The first assertion follows from [\cite{Boecherer-Schmidt00}, Section 5]. Let $B \in \calh_{2n}(\ZZ)_{> 0}$. Then, 
\begin{align*}
&c_{2n,l}(B,N,\phi)=(-1)^{nl}2^{2n}\Gamma(l-n)\\
&\times (\det (2B))^{l-n-1/2}\prod_p F_p(B,p^{-l}\phi(p)){L(l-n,\phi\chi_B) \over \pi^{l-n}}.
\end{align*}
We have 
\[L(l-n,\phi\chi_B)=L(l-n,(\phi\chi_B)_0)\prod_{p | N|\frkd_B|} (1-p^{n-l}(\phi\chi_B)_0),\]
and
\begin{align*}
&{\Gamma(l-n)L(l-n,(\phi\chi_B)_0) \over \pi^{l-n}}=(-1)^{(l-n-\delta_{(\phi\chi_B})_0)/2}2^{l-n-1} m_{(\phi\chi_B)_0}^{n-l}\sqrt{-1}^{-\delta_{(\phi\chi_B)_0}}\\
&\times L(1-l+n,\overline{(\phi\chi_B)_0}).
\end{align*}
Moreover, by the functional equation of $F_p(B,X)$ (cf. \cite{Katsurada99}), we have
\[\frkf_B^{2l-2n-1}\prod_p F_p(B,p^{-l}\phi(p))=\prod_p F_p(p^{l-2n-1}\bar \phi(p),B).\]
Thus the assertion is proved remarking that $\det (2B)=|\frkd_B|\frkf_B^2$.
\end{proof}

\begin{corollary} \label{cor.FC-Siegel-Eisenstein}
Let the notation be as above. 
\begin{itemize}
\item [(1)] Suppose that $N=1$. 
Then, $c_{2n,l}(B,1,{\bf 1})$ belongs to \\ $\langle (\prod_{i=1}^n (2l-2i)(2l-2i+1)!)^{-1}\rangle_{\ZZ}$
 for any $B \in \calh_{2n}(\ZZ)_{\ge 0}$. 
\item [(2)] Suppose that $N>1$. Then for $B \in \calh_{2n}(\ZZ) {\ge 0}$, $c_{2n,l}(B,N,\phi)$ is an algebraic number. 
In particular if  $\mathrm{GCD}(\det (2B),N)=1$, then
$\tau(\phi)^{-1}\sqrt{-1}^{-l}c_{2n,l}(B,N,\phi)$ belongs to $\langle (l-n)^{-1}\rangle_{\frkO_{\QQ(\phi)}[N^{-1}]}$.
\end{itemize}
 \end{corollary}
\begin{proof}
(1) By  Proposition \ref{prop.precise-Kitaoka-polynomial}, the product $\prod_{p | \det (2\widetilde B)} F_p^\ast(B,l-m-1)$  is an integer for any $m$ and $B \in \calh_n(\ZZ)_{\ge 0}$ with rank $m$. 
By Clausen-von-Staudt theorem, $\zeta(1-2l+2i)$ belongs to $\langle( (2l-2i)(2l-2i+1)!)^{-1} \rangle_{\ZZ}$.
By [\cite{Boecherer84}, (5.1), (5.2)] and Clausen-von-Staudt theorem, for any positive even integer $m$ and $\widetilde B \in \calh_m(\ZZ)_{>0}$, 
$L(1-l+m/2,\chi_{\widetilde B})$ belongs to $\langle ((2l-m)(2l-m+1)!)^{-1}\rangle_{\ZZ}$.
This proves the assertion.\\
(2) It is well known that $L(1-l+n,\overline{(\phi \chi_B)_0})$ is algebraic. This proves the first part of the assertion. 
Suppose that $\det (2B)$ is coprime to $N$. Then $\phi \chi_B$ is a primitive character of conductor $N|\frkd_B|$ and
\begin{align*}
\tau(\phi \chi_B)&=\phi(|\frkd_B|)\chi_B(N)\tau(\phi)\tau(\chi_B)\\
&=\phi(|\frkd_B|)\chi_B(N)\tau(\phi)|\frkd_B|^{1/2}\sqrt{-1}^{\delta_{\chi_B}}.
\end{align*}
By \cite{Carlitz59} or \cite{Leopoldt58},  $N(l-n)L(1-l+n,\overline{\phi \chi_B})$ belongs to $\frkO_{\QQ(\phi)}$,
and  by Proposition \ref{prop.precise-Kitaoka-polynomial}, $\prod_p F_p(p^{l-2n-1} \bar \phi(p),B)$ is an element of 
$\frkO_{\QQ(\phi)}$. Thus the assertion has been proved remarking that $\sqrt{-1}^l=\pm \sqrt{-1}^{\delta_{\chi_B}-\delta_{\phi\chi_B}}$.
\end{proof}

Let $\stackrel {\!\!\!\!\!\! \circ} {{\mathcal D}_{n,l}^{\nu}}$ be the differential operator in \cite{Boecherer-Schmidt00}, which maps 
$M_{l}^{\infty}(\varGamma^{(2n)}_0(N))$ to $M_{l+\nu}^{\infty}(\varGamma^{(n)}_0(N)) \otimes M_{l+\nu}^{\infty}(\varGamma^{(n)}_0(N)).$  Let $\chi$ be a primitive character mod $N.$ For a non-negative integer $\nu \le k$, we define a function 
${\mathfrak  E}_{2n}^{k,\nu}(Z_1,Z_2,N,\chi)$ on ${\bf H}_n \times {\bf H}_n$ as
\begin{align*}
&{\mathfrak  E}_{2n}^{k,\nu}(Z_1,Z_2,N,\chi) = (2\pi \sqrt{-1})^{-\nu}\tau(\chi)^{-n-1}\sqrt{-1}^{-k+\nu}\\
& \times \stackrel {\!\!\!\!\!\! \circ} {{\mathcal D}_{n,k-\nu}^{\nu}}\left(\sum_{X \in M_n(\ZZ)/NM_n(\ZZ)} \chi(\det X)E^*_{2n,k-\nu}(*,N,\chi)|_{k-\nu}
\left(\begin{smallmatrix} 1_{2n} & S(X/N)\\O&1_{2n}\end{smallmatrix}\right)\right)(Z_1,Z_2)
\end{align*}
for $(Z_1,Z_2) \in {\bf H}_n \times {\bf H}_n,$ where $S(X/N)=\mattwo(O_n;X/N;{}^tX/N;O_n).$ 
Let $X$ be a symmetric matrix of size $2n$ of variables.
Then there exists a polynomial $P_{n,l}^{\nu}(X)$ in $X$ such that
\begin{align*} &\stackrel {\!\!\!\!\!\! \circ} {{\mathcal D}_{n,l}^{\nu}}\Bigl({\bf e}\bigl(\mathrm{tr}\Bigl(\begin{pmatrix} A_1 & R/2 \\ {}^tR/2 & A_2 \end{pmatrix} \begin{pmatrix} Z_1 & Z_{12} \\ {}^t Z_{12} & Z_2 \end{pmatrix}\Bigr)\Bigr)\Bigr)\\
&=(2\pi \sqrt{-1})^{\nu}P_{n,l}^{\nu}\Bigl(\begin{pmatrix} A_1 & R/2 \\ {}^tR/2 & A_2 \end{pmatrix}\Bigr) 
{\bf e}(\mathrm{tr}(A_1Z_1+A_2Z_2)) 
\end{align*}
for $\begin{pmatrix} A_1 & R/2 \\ {}^tR/2 & A_2 \end{pmatrix} \in \calh_{2n}(\ZZ)_{\ge 0}$ with $A_1,A_2 \in \calh_n(\ZZ)_{\ge 0}$
and $ \begin{pmatrix} Z_1 & Z_{12} \\ {}^t Z_{12} & Z_2 \end{pmatrix} \in {\bf H}_{2n}$ with $Z_1,Z_2 \in {\bf H}_n$.
\begin{proposition} \label{prop.FC-Siegel-Eisenstein2}
Under the above notation and the assumption, for a non-negative integer $l \le k$ write
${\mathfrak  E}_{2n}^{k,k-l}(Z_1,Z_2,N,\chi)$
as
\begin{align*}
{\mathfrak  E}_{2n}^{k,k-l}(Z_1,Z_2,N,\chi) =\sum_{A_1,A_2 \in \calh_n(\ZZ)_{\ge 0}} c_{{\mathfrak  E}_{2n}^{k,k-l}(Z_1,Z_2,N,\chi)}(A_1,A_2){\bf e}(\mathrm{tr}(A_1Z_1+A_2Z_2)
\end{align*}
Then we have
\begin{align*}
&c_{{\mathfrak  E}_{2n}^{k,k-l}(Z_1,Z_2,N,\chi)}(A_1,A_2)\\
&=
\sum_{R \in M_n(\ZZ)} P_{n,l}^{k-l}\Bigl(\begin{pmatrix} A_1 & R/2 \\ {}^tR/2 & A_2 \end{pmatrix}\Bigr)c_{2n,l}\Bigl(\begin{pmatrix} A_1 & R/2 \\ {}^tR/2 & A_2 \end{pmatrix}\Bigr)\bar \chi(\det R)\tau(\chi)^{-1}\sqrt{-1}^{-l}
\end{align*}
\end{proposition}
\begin{corollary} \label{cor.FC-Siegel-Eisenstein2}
For any $A_1,A_2 \in \calh_n(\ZZ)_{>0}$, 
$c_{{\mathfrak  E}_{2n}^{k,k-l}(Z_1,Z_2,N,\chi)}(A_1,A_2)$ belongs to $\bar \QQ$, and in particular if $\det \Bigl(\begin{pmatrix} 2A_1 & R \\ {}^tR & 2A_2 \end{pmatrix}\Bigr)$ is prime to $N$, then $a_{n,l}c_{{\mathfrak  E}_{2n}^{k,k-l}(Z_1,Z_2,N,\chi)}(A_1,A_2)$ belongs $\frkO_{\QQ(\chi)}[N^{-1}]$,
where $a_{n,l}=\prod_{i=1}^n (2l-2i)(2l-2i+1)!$.
\end{corollary}
\bigskip
Suppose that $l \le k.$ Then ${\mathfrak  E}_{2n}^{k,k-l}(Z_1,Z_2,N,\chi)$ can be expressed as
$${\mathfrak  E}_{2n}^{k,k-l}(Z_1,Z_2,N,\chi)=\sum_{A \in {\mathcal L}_n(\ZZ)_{>0}}{\mathcal E}_{2n}^{k,k-l}(Z_1,A,N,\chi){\bf e}({\rm tr}(AZ_2))$$
with ${\mathcal E}_{2n}^{k,k-l}(Z_1,A,N,\chi)$ a  function of $Z_1$.  Put
 $${\mathcal G}_{2n}^{k,k-l}(Z_1,A,N,\chi)=\sum_{\gamma \in \varGamma_0^{(n)}(N^2) \backslash \varGamma^{(n)}}( {\mathcal E}_{2n}^{k,k-l})|_k\gamma(Z_1,A,N,\chi).$$
It is easily seen that ${\mathcal E}_{2n}^{k,k-l}(Z_1,A,N,\chi)$  belongs to $M_k(\varGamma_0^{(n)}(N^2))$, and therefore
 ${\mathcal G}_{2n}^{k,k-l}(Z_1,A,N,\chi)$ belongs to 
$M_k(\varGamma^{(n)})$. In particular, if $l<k$, then ${\mathcal G}_{2n}^{k,k-l}(Z_1,A,N,\chi)$ belongs to 
$S_k(\varGamma^{(n)})$.
\begin{proposition} \label{prop.integrality-of-partial-Eisen} 
Suppose that  $l \le k$ and let $A \in \calh_n(\ZZ)_{>0}$. 
Then  $a_{n,l}\calg_{2n}^{k,k-l}(Z_1,N^2A,N,\chi)$ belongs to $M_k(\varGamma^{(n)})({\mathcal O}_{\QQ(\chi,\zeta_N)}[N^{-1}]).$ In particular, if $l<k$, it belongs to $S_k(\varGamma^{(n)})({\mathcal O}_{\QQ(\chi,\zeta_N)}[N^{-1}]).$
\end{proposition}
\begin{proof}
We have
\begin{align*}
&c_{{\mathfrak  E}_{2n}^{k,k-l}(Z_1,Z_2,N,\chi)}(B,N^2A)\\
&=\sum_{R  \in M_n(\ZZ)} P_{n,l}^{k-l}\Bigl(\begin{pmatrix} B & R/2 \\ {}^tR/2 &  N^2A  \end{pmatrix}\Bigr)c_{2n,l}\Bigl(\begin{pmatrix} B & R/2 \\ {}^tR/2 & N^2A \end{pmatrix}\Bigr)\bar \chi(\det R)\tau(\chi)^{-1}\sqrt{-1}^{-l}.
\end{align*}
We note that  $\det \begin{pmatrix} 2B & R \\ {}^tR  & 2N^2A \end{pmatrix}$ is prime to $N$ if and only $\det R$ is prime to $N$. Therefore, by Corollary \ref{cor.FC-Siegel-Eisenstein},
 $a_{n,l}\cale_{2n}^{k,k-l}(Z_1,N^2A,N,\chi)$ belongs to $M_k(\varGamma_0^{(n)}(N^2)) (\frkO_{\QQ(\chi)}[N^{-1}])$.
By $q$-expansion principle (cf. \cite{Ichikawa14}, \cite{Katz73}),
 for any $\gamma \in \varGamma^{(n)}$,
$a_{n,l}\cale_{2n}^{k,k-l}|_k\gamma(Z_1,N^2A,N,\chi)$ belongs to 
$M_k(\varGamma^{(n)}(N^2)) (\frkO_{\QQ(\chi,\zeta_N)}[N^{-1}])$.
Hence, $a_{n,l}\calg_{2n}^{k,k-l}(Z_1,N^2A,N,\chi)$ belongs to
$M_k(\varGamma^{(n)}(N^2)) (\frkO_{\QQ(\chi,\zeta_N)}[N^{-1}]) \cap M_k(\varGamma^{(n)})=M_k(\varGamma^{(n)})(\frkO_{\QQ(\chi,\zeta_N)}[N^{-1}])$. This proves the first of the assertion. The latter is similar.
\end{proof}

\begin{theorem}
\label{th.pullback-Eisenstein}
Let $\{F_i \}_{i=1}^d$ be an orthogonal basis of $ S_k(\varGamma^{(n)})$ consisting of Hecke eigenforms, and $\{F_i \}_{d+1 \le i \le e}$ be a basis of the orthogonal complement $S_k(\varGamma^{(n)})^\perp$ of $S_k(\varGamma^{(n)})$ in $M_k(\varGamma^{(n)})$ with respect to the Petersson product.
Then we have
\begin{align*}
{\mathcal G}_{2n}^{k,k-l}(Z,N^2A,N,\chi) &=\sum_{i=1}^d c(n,l)N^{nl}\Lambda(l-n,F_i, \chi, {\rm St}) \overline{c_{F_i}(A)}F_i(Z)\\
&+ \sum_{i=d+1}^e c_i F_i(Z)
\end{align*}
where  $c(n,l)=(-1)^{a(n,l)}2^{b(n,l)}$ with $a(n,l),b(n,l)$ integers, and $c_i$ is a certain complex number. Moreover we have $c_i=0$ for any $d+1 \le i \le e$ if $l<k$.
\end{theorem}
\begin{proof} Put 
\[{\mathfrak  G}_{2n}^{k,k-l}(Z_1,Z_2,N,\chi)= \sum_{\gamma \in \varGamma_0^{(n)}(N^2) \backslash \varGamma^{(n)}}{\mathfrak  E}_{2n}^{k,k-l}(|_k\gamma Z_1,Z_2,N,\chi).\]
Then we have
\begin{align*}
{\mathfrak  G}_{2n}^{k,k-l}(Z_1,Z_2,N,\chi)&=\sum_{A \in {\mathcal L}_n(\ZZ)_{>0}}{\mathcal G}_{2n}^{k,k-l}(Z_1,A,N,\chi){\bf e}({\rm tr}(AZ_2))
\end{align*}
By [\cite{Boecherer-Schmidt00},(3.24)], for any $\gamma \in \varGamma^{(n)}$ we have
\begin{align*}
&\langle F_i, {\mathfrak  G}_{2n}^{k,k-l}(|_k\gamma \     *,-\overline{Z_2},N,\chi) \rangle\\
=&\langle F_i|_k\gamma, {\mathfrak  G}_{2n}^{k,k-l}(|_k\gamma  \  *,-\overline{Z_2},N,\chi) \rangle\\
=&\langle F_i, {\mathfrak  G}_{2n}^{k,k-l}( *,-\overline{Z_2},N,\chi) \rangle\\
=&(-1)^{a'(n,l)}2^{b'(n,l)}N^{nl}\chi(-1)^n [\varGamma^{(n)}: \varGamma_0^{(n)}(N^2)]^{-1}\pi^{(l-k)n-(2n+1)l+n} \pi^{n(n+1)/2}\\
&\times L(l-n,F_i,\bar \chi,\mathrm{St}) \Gamma(l-n)\tau(\chi)^{-n-1}\sqrt{-1}^{-l}F_i(N^2Z_2)\\
&\times {\Gamma_{2n}(l) \Gamma_n(k-n/2)\Gamma_n(k-(n+1)/2) \over \Gamma_n(l)\Gamma_n(l-n/2)},
\end{align*}
with $a'(n,l),b'(n,l) \in \ZZ$. 
 We note that we take the  normalized Petersson inner product. We also note that 
\[\Gamma_{2n}(l)=\pi^{n^2/2}\Gamma_n(l)\Gamma_n(l-n/2),\]
and
\[\Gamma_n(k-n/2)\Gamma_n(k-(n+1)/2)=2^{\gamma'(n,l)}\pi^{n^2/2}\prod_{i=1}^n \Gamma(2k-n-i)\]
with an integer $\gamma'(n,l)$.
Hence we have 
\begin{align*}
&\langle F_i, {\mathfrak  G}_{2n}^{k,k-l}(|_k\gamma *,-\overline{Z_2},N,\chi) \rangle\\
&=c(n,l)[\varGamma^{(n)}: \varGamma_0^{(n)}(N^2)]^{-1} N^{nl}\Lambda(l-n,F_i,\overline{\chi},\St)\langle F_i,F_i \rangle F_i(N^2Z_2), 
\end{align*}
where  $c(n,l)=(-1)^{a(n,l)}2^{b(n,l)}$ with $a(n,l),b(n,l)$ integers.
On the other hand, we have 
$$\langle F_i, {\mathfrak  G}_{2n}^{k,k-l}(*,-\overline{Z_2},N,\chi) \rangle=\sum_{A \in {\mathcal L}_n(\ZZ)_{>0}}
\langle F_i, {\mathcal G}_{2n}^{k,k-l}(*,A,N,\chi) \rangle {\bf e}({\rm tr}(AZ_2)).$$
Hence we have
$$\langle F_i, {\mathcal G}_{2n}^{k,k-l}(*,A,N,\chi) \rangle=c(n,l) N^{nl}\Lambda(l-n,F_i,\overline{\chi},\St)\langle F_i,F_i \rangle c_{F_i}(N^{-2}A)$$
for any $A.$ Now ${\mathcal G}_{2n}^{k,k-l}(Z,A,N,\chi) $ can be expressed as
$${\mathcal G}_{2n}^{k,k-l}(Z,A,N,\chi)=\sum_{i=1}^e c_iF_i(Z) $$ 
with $c_i \in \CC.$ For $1 \le i \le d$  we have
$$\langle F_i,{\mathcal G}_{2n}^{k,k-l}(*,A,N,\chi)\rangle=\overline {c_i} \langle F_i,F_i \rangle.$$
Hence we have
$$c_i=\overline{c(n,l)N^{nl}  \Lambda(l-n,F_i,\overline{\chi},\St)\langle F_i,F_i \rangle c_{F_i}(N^{-2}A)}.$$
We note that $\overline{ \Lambda(l-n,F_i,\overline{\chi},\St)}= \Lambda(l-n,F_i,\chi,\St).$ This proves the assertion.
\end{proof}
\begin{remark}
\label{rem.corrections}
There are errors in {\rm [\cite{Katsurada15}, Appendix]}. \\
{\rm (1)} The factor $\eta^*(\gamma)$  is missing in $E_{n,l}(Z,M,\eta,s)$ on {\rm [\cite{Katsurada15}, page 125]}, and it should be defined as 
\[E_{n,l}(Z,M,\eta,s)=L(1-l-2s,\eta)\prod_{i=1}^{[n/2]} L(1-2l-4s+2i,\eta^2)$$
$$\times  \det ({\rm Im}(Z)))^s \sum_{\gamma \in \varGamma_{\infty}^{(n)} \backslash \varGamma_0^{(n)}(M)} j(\gamma,Z)^{-l} \eta^*(\gamma)|j(\gamma,Z)|^{-2s}.\]
Then $E^*_{n,l}(Z,M,\eta,s)=E_{n,l}|_l W_M(Z,M,\eta,s)$
with $W_M=\mattwo(O;-1_n;M1_n;O)$ coincides with the Eisenstein series $E^*_{n,l}(Z,M,\eta,s)$ in the present paper 
up to elementary factor. However, to quote several results 
in \cite{Boecherer-Schmidt00} smoothly, we define $E^*_{n,l}(Z,M,\eta,s)$ as in the present paper. Accordingly we define 
${\mathcal G}_{2n}^{k,k-l}(Z,A,N,\chi)$ as in our paper.
 With these changes,
Propositions 5.1 and 5.2, and (1) of Theorem 5.3 in \cite{Katsurada15} should be replaced with 
Corollary \ref{cor.FC-Siegel-Eisenstein}, Corollary \ref{cor.FC-Siegel-Eisenstein2}, and Proposition \ref{prop.integrality-of-partial-Eisen}, respectively, in the present paper. \\
{\rm (2}) In \cite{Katsurada15}, we defined ${\bf L}(m,F,\chi,\St)$ as 
\[{\bf L}(m,F,\chi,\St)= \Gamma_{\CC}(m)(\prod_{i=1}^n \Gamma_{\CC}(m+k-i)) {L(m,F,\chi,\St) \over \tau(\chi)^{n+1}\langle F, \ F \rangle},\]
where $\Gamma_{\CC}(s)=2(2\pi)^{-s}\Gamma(s)$.
However, the factor  $\sqrt{-1}^{m+n}$ should be added in the denominator on the right-hand side of the above definition. With this correction,
{\rm [\cite{Katsurada15}, Theorem 2.2]} remains valid. 
Moreover, we have 
\begin{align*}
&{\bf L}(l-n,F,\chi,\St)\\
&= {\prod_{i=1}^n  \Gamma_{\CC}(l-n+k-i) \over N^{ln}c(n,l)
\prod_{i=1}^n \Gamma(2k-n-i) \pi^{-n(n+1)/2+nk+(n+1)m}}\Lambda(l-n,F,\chi,\St).
\end{align*}
We note that
\[{\displaystyle \prod_{i=1}^n  \Gamma_{\CC}(l-n+k-i) \over \displaystyle N^{ln}c(n,l)
\prod_{i=1}^n \Gamma(2k-n-i) \pi^{-n(n+1)/2+nk+(n+1)m}}\]
 is a rational number, and for a prime number $p$ not dividing $N(2k-1)!$,
it is $p$-unit. Therefore,  (2) of Theorem 5.3 in \cite{Katsurada15} should be corrected as follows:

\bigskip

{ \it Put 
\begin{align*}
&\widetilde {\mathcal G}_{2n}^{k,k-l}(Z,N^2A,N,\chi)\\
&= {\prod_{i=1}^n  \Gamma_{\CC}(l-n+k-i) \over N^{ln}c(n,l)
\prod_{i=1}^n \Gamma(2k-n-i) \pi^{-n(n+1)/2+nk+(n+1)m}}\\
& \times {\mathcal G}_{2n}^{k,k-l}(Z,N^2A,N,\chi).
\end{align*}
Then $\widetilde {\mathcal G}_{2n}^{k,k-l}(Z,N^2A,N,\chi)$ belongs to $\frkS_k(\varGamma^{(n)})(\frkO_{\QQ(F,\chi,\zeta_N)})_{\frkP}$ for any prime ideal $\frkP$ of $\QQ(F,\chi,\zeta_N)$ not dividing $N(2k-1)!$, 
and  we have
$$\widetilde {\mathcal G}_{2n}^{k,k-l}(Z,N^2A,N,\chi) =\sum_{i=1}^d {\bf L}(l-n,F_i, \chi, {\rm St}) \overline{c_{F_i}(A)}F_i(Z).$$}
 
\end{remark}

\section{Proof of the main result}

\begin{lemma}
\label{lem.fundamental-lemma}
Let $r \ge 2$ and let $\{F_1,\ldots,F_r \}$ be Hecke eigenforms   $M_{k}(\varGamma^{(n)};\lambda_i)$ linearly independent over $\CC,$ and $G$ an element of $M_{k}(\varGamma^{(n)}).$ Write
$$F_{i}(Z)=\sum_{A} c_{F_{i}}(A) {\bf e}({\rm tr}(AZ))$$
for $i=1,...r$ and
$$G(Z)=\sum_{A} c_G(A) {\bf e}({\rm tr}(AZ)).$$
Let $K$ be the composite field of $\QQ(F_1),\ldots, \QQ(F_r)$,  and $L$ a finite extension of $K$. Let $N$ be a positive integer. Assume that
 \begin{itemize}
\item[(1)] there exists an element $\alpha \in K$ such that $c_G(A)$ belongs to $\alpha {\frkO}_{L}[N^{-1}]$ for any $A \in {\mathcal H}_n(\ZZ)_{>0}$ 
\item[(2)] there exist $c_{i} \in L \ (i=1,...,r)$ and $A \in {\mathcal H}_{n}(\ZZ)_{>0} $ such that 
 $$G(Z)=\sum_{i=1}^r c_{i} F_{i}(Z).$$
\end{itemize}
Then for  any elements $T_1,\ldots,T_{r-1} \in {\bf L}_n'$ and 
$A \in {\mathcal H}_{n}(\ZZ)_{>0}$ we have
$$\prod_{i=1}^{r-1}(\lambda_{F_1}(T_i)-\lambda_{F_{i+1}}(T_i)) c_{1}c_{F_1}(A) \in \alpha  {\frkO}_{L}[N^{-1}].$$
\end{lemma}

\begin{proof}
We prove the induction on $r$. 
The assertion clearly holds for $r=2$. Let $r \ge 3$ and suppose that the assertion holds for any $r'$ such that $2 \le r' \le r-1$.
We have 
 $$G|T_{r-1}(Z)=\sum_{i=1}^r\lambda_{F_i}(T_{r-1}) c_{i}  F_{i}(Z),$$
and we have
$$G|T_{r-1}(Z) - \lambda_{F_r}(T_{r-1})G(Z)=\sum_{i=1}^{r-1}(\lambda_{F_i}(T_{r-1})-\lambda_{F_r}(T_{r-1})) c_{i}  F_{i}(Z).$$
By Theorem 4.1 and Proposition 4.2 of \cite{Katsurada08}, we have
$$G|T_{r-1}(Z) - \lambda_{T_{r-1}}G(Z) \in \alpha S_{k}(\varGamma^{(n)})({\frkO}_{L}[N^{-1}])$$
Hence, by the induction assumption we prove the assertion.

\end{proof}

\bigskip

{\bf Proof of Theorem \ref{th.main-result}}
Let $b(n,l)$ be the integer in Theorem \ref{th.pullback-Eisenstein}, 
and put $\alpha(n,k)=\max_{2 \le l \le k-n-2 \atop  l \equiv 0 \text{ mod } 2} b(n,l)$. Then, $a_{n,l}{\mathcal G}_{2n}^{k,k-l}(Z,N^2A,N,\chi) \in 2^{-\alpha(n,k)}M_{k}(\varGamma^{(n)})({\frkO}_{\QQ(\chi,\zeta_N)}[N^{-1}])$. Thus, by   Theorem \ref{th.pullback-Eisenstein} and Lemma \ref{lem.fundamental-lemma}, for any $B  \in \calh_n(\ZZ)_{>0}$,
and $T_1,\ldots,T_e \in {\bf L}_n'$, the value
$$\prod_{i=1}^{r-1}(\lambda_{F_1}(T_i)-\lambda_{F_{i+1}}(T_i))
\Lambda(l-n,F,\chi,\St)\bar c_F(A)c_F(B) $$
belongs to $(2^{\alpha(n,k)}A_{n,k})^{-1} {\frkO}_{L_{n,k}(\chi,\zeta_N)}[N^{-1}],$
where $e=\dim_\CC M_{k}(\varGamma^{(n)})$, and $L_{n,k}$ is the field stated in Section 1. In particular for any
$v \in \widetilde \frkE_{F}$, the value
$v \Lambda(l-n,F,\chi,\St)\bar c_F(A)c_F(B)$
belongs to $(2^{\alpha(n,k)}A_{n,k})^{-1} {\frkO}_{L_{n,k}(\chi,\zeta_N)}[N^{-1}].$
On the other hand, by Proposition \ref{prop.algebraicity-standard-L}, the value $\Lambda(l-n,F,\chi,\St)\bar c_F(A)c_F(B)$ belongs to $\QQ(F,\chi)$, and hence
we have 
$$v \Lambda(l-n,F,\chi,\St)\bar c_F(A)c_F(B)  \in (2^{\alpha(n,k)}A_{n,k})^{-1} {\frkO}_{\QQ(F,\chi)}[N^{-1}].$$
This implies that we have
\[  \frkI(l-n,F,\chi) \subset  \langle (2^{\alpha(n,k)} A_{n,k} \widetilde \frkE_{F})^{-1}\rangle_{\frkO_{\QQ(F,\chi)}[N^{-1}]}.\]
\begin{remark} \label{rem.general-congruence}
 Let the notation be as in Lemma \ref{lem.fundamental-lemma}. Then we have the following.

Let $\frkp$ be a prime ideal of $K$. Assume that $c_1 c_{F_1}(A)$ belongs to $K$ and that $\ord_\frkp(c_1c_{F_1}(A))<0$ for some $A \in \calh_n(\ZZ)_{>0}$. Then there exists $i \not=2$ such that we have
\[\lambda_{F_i}(T) \equiv \lambda_{F_1}(T) \text{ mod } \frkp \quad 
\text{ for any } T \in {\bf L}_n'.\]

This is a generalization of \rm{[\cite{Katsurada08}, Lemma 5.1]}, and it can be proved in the same way. Let $K_{n,k}$ be the field defined in Section 2. Then, applying the above result to $L=K_{n,k}(\chi,\zeta_N)$, and using a corrected version of [\cite{Katsurada15}, Theorem 5.3] in Remark \ref{rem.corrections} (2), we can remedy the proof of  [\cite{Katsurada15}, Theorem 3.1].

We also remark that the $M(2l-1)!$ in {\rm [\cite{Katsurada15}, Theorem 3.1]} should be $M(2k-1)!$.
\end{remark}

\section{Boundedness of special values of products of Hecke $L$-functions}
For an element $f(z)=\sum_{m=1}^\infty  c_f(m){\bf e}(mz) \in S_k(SL_2(\ZZ))$ and a Dirichlet character $\chi$, we define Hecke's $L$ function $L(s,f,\chi)$ as
\[L(s,f,\chi)=\sum_{m=1}^\infty {c_f(m) \over m^s}.\]
Let $f$ be a primitive form. Then, for two positive integers $l_1,l_2 \le k-1$ and Dirichlet characters $\chi_1,\chi_2$ such that $\chi_1(-1)\chi_2(-1)=(-1)^{l_1+l_2+1},$ the value 
$${\Gamma_{\CC}(l_1)\Gamma_{\CC}(l_2)L(l_1,f,\chi_1)L(l_2,f,\chi_2) \over \sqrt{-1}^{l_1+l_2+1}\tau((\chi_1 \chi_2)_0) \langle f , \ f \rangle}$$
belongs to $\QQ(f, \chi_1, \chi_2)$
 (cf. \cite{Shimura76}). We denote this value by ${\bf L}(l_1,l_2;f;\chi_1,\chi_2).$ In particular, we put
\[{\bf L}(l_1,l_2;f)={\bf L}(l_1,l_2;f;\chi_1,\chi_2)\]
if $\chi_1$ and $\chi_2$ are  the principal characters.

\begin{theorem}
\label{th.main-result-2}
Let $f$ be a primitive form in $S_k(SL_2(\ZZ))$. 
Then  we have
\[ {\bf L}(l_1,l_2;f;\chi_1,\chi_2) \in  \langle (2^{b_k}\zeta(1-k)(k!)^2\widetilde \frkD_{f})^{-1} \rangle_{\frkO_{\QQ(f,\chi_1,\chi_2)}[(N_1N_2)^{-1}]}\]
with some non-negative integer $b_k$
 for any integers $l_1$ and $\l_2$ and primitive characters $\chi_1$ and $\chi_2$ of conductors $N_1$ and $N_2$, respectively, satisfying the following conditions: 
\begin{align*}
 (\chi_1\chi_2)(-1)=(-1)^{l_1+l_2+1}. \tag{D1} \\
 k-l_1+1 \le l_2 \le l_1-1 \le k-2 \tag{D2}\\
\text{ Either } l_1 \ge l_2+2, \text{ or } l_1=l_2+1 \text{ and } \chi_1 \text{ or } \chi_2 \text{ is  non-trivial} \tag{D3}
\end{align*} 
\end{theorem}
\begin{proof}
 The proof will proceed by a careful analysis of the proof of [\cite{Shimura76}, Theorem 4] combined with the argument in Theorem \ref{th.main-result}. 
For a positive integer $\lambda \ge 2$ and 
a Dirichlet character $\omega$  mod $N$ such that $\omega(-1)=(-1)^\lambda$ we define the Eisenstein series
$G_{\lambda,N}(z,s,\omega) \ ( z \in {\bf H}_1, s \in \CC)$ by
\[G_{\lambda,N}(z,s,\omega)=\sum_{\gamma \in \varGamma_{\infty} \backslash \varGamma^{(1)}_0(N)} \omega(d) (cz+d)^{-\lambda}|cz+d|^{-2s} \qquad \gamma=\begin{pmatrix} a & b \\ c & d \end{pmatrix},\]
where $\varGamma_\infty=\{\pm \begin{pmatrix} 1 & m \\ 0 & 1 \end{pmatrix} \ \Bigl| \ m \in \ZZ \}.$
It is well known that $G_{\lambda,N}(z,s,\omega)$ is finite at $s=0$ as a function of $s$, and put
\[G_{\lambda,N}(z,\omega)=G_{\lambda,N}(z,0,\omega).\]
$G_{\lambda,N}(z,\omega)$ is a (holomorphic) modular form of weight $\lambda$ and character $\bar \omega$ for $\varGamma^{(1)}_0(N)$  if  $\lambda \ge 3$ or $\omega$ is non-trivial. In the case $\lambda=2$ and $\omega$ is trivial,  $G_{2,N}(z,\omega)$ is a nearly automorphic form of weight $2$ for $\varGamma^{(1)}_0(N)$ in the sense of Shimura \cite{Shimura86}. We also put 
\[\widetilde G_{\lambda,N}(z,\omega)={2\Gamma(\lambda) \over (-2\pi \sqrt{-1})^{\lambda}\tau(\omega_0)}L_N(\lambda,\omega)G_{\lambda,N}(z,\omega),\]
where  $L_N(s,\omega)=L(s,\omega)\prod_{p|N}(1-p^{-s}\omega(p))$.
  Now let $N_i$ be the conductor of $\chi_i$ for $i=1,2$. Then, by [\cite{Miyake89}, Theorem 4.7.1] there exists
a modular form $g$ of weight $l_1-l_2+1$ and character $\chi_1\chi_2$ for $\varGamma_0^{(1)}(N_1N_2)$ such that
\begin{align*} c_g(0)&=\begin{cases} 0 & \text{ if } \chi_1 \text{ is non-trivial } \\
{-1(1-N_1N_2) \over 24}  & \text{ if } l_1-l_2=1 \text{ and both } \chi_1 \text{ and } \chi_2 \text{ are trivial } \\
{-B_{l_1-l_2+1,\chi_1\chi_2} \over 2(l_1-l_2+1)} & \text{ otherwise},
\end{cases} \\
c_g(m)&=\sum_{0 < d|m} \chi_1(m/d)\chi_2(d)d^{l_1-l_2} \quad (m \ge 1), \end{align*}
and 
\[L(s,g)=L(s,\chi_1)L(s-l_1+l_2,\chi_2).\]
Since we have $k \ge l_2,l_1$,  all the Fourier coefficients of $g$ belong to $(k!)^{-1}\frkO_{\QQ(\chi_1,\chi_2)}[(N_1N_2)^{-1}]$. Put $\lambda=-k+l_1+l_2+1$. 
Let $\delta_\lambda^{(r)}$ be the differential operator in \cite{Shimura76}, page 788. Then, [\cite{Shimura76}, Lemma 7] we have
\[g \delta_{-k+l_1+l_2+1}^{(k-l_1-1)}\widetilde G_{-k+l_1+l_2+1,N_1N_2}(z,\chi_1\chi_2)=\sum_{\nu=0}^r \delta_{k-2\nu}^{(\nu)} h_\nu(z)\]
with some $r<k/2$, and $h_\nu \in M_{k-2\nu}(\varGamma^{(1)}_0(N_1N_2))$.
By [\cite{Shimura76}, (3.3) and (3.4)] and the assumption, all the Fourier coefficients of  $\widetilde G_{-k+l_1+l_2+1,N_1N_2}(z,\chi_1\chi_2)$ belongs to $(k!)^{-1}\frkO_{\QQ(\chi_1\chi_2)}[(N_1N_2)^{-1}]$ if 
$-k+l_1+l_2+1 \ge 3$, or  $\chi_1\chi_2$ is  non-trivial. Moreover, by  [\cite{Shimura76}, page 795],  
$\widetilde G_{2,N_1N_2}(z,\chi_1\chi_2)$ is expressed as
\[\widetilde G_{2,N_1N_2}(z,\chi_1\chi_2)={c \over 4\pi y} +\sum_{n=0}^\infty c_n{\bf e}(nz),\]
with $c,c_n \in 2^{-1}\frkO_{\QQ(\chi_1\chi_2)}[(N_1N_2)^{-1}]$   if $-k+l_1+l_2+1=2$ and $\chi_1\chi_2$ is trivial. Hence, by the construction of $h_{0}$,   all the Fourier coefficients of $h_{0}$ belong to $((k!)^2)^{-1}\frkO_{\QQ(\chi_1,\chi_2)}[(N_1N_2)^{-1}]$.
 Let $f_1,\ldots,f_d$ be a basis of $S_k(SL_2(\ZZ))$ consisting of primitive forms such that $f_1=f$. Then, by [\cite{Shimura76}, Theorem 2, Lemmas 1 and 7], we have
\[{\bf L}(l_1,l_2,f_i;\chi_1,\chi_2)\langle f_i, \ f_i \rangle=d_0[SL_2(\ZZ):\varGamma^{(1)}_0(N_1N_2)]\langle f,h_{0} \rangle \]
for any $i=1,\ldots,d$,
where $d_0=(-1)^{a(k,l_1,l_2)}2^{b(k,l_1,l_2)}$ with some $a(k,l_1,l_2),b(k,l_1,l_2) \in \ZZ$. 
(We note that  the Petersson product $\langle *, \ * \rangle$ in our paper is ${\pi \over 3}$ times that  in \cite{Shimura76}).
Define ${\bf h}_{0}(z)$ by
\[{\bf h}_{0}=d_0\sum_{\gamma \in \varGamma^{(1)}_0(N_1N_1) \backslash SL_2(\ZZ)} h_{0}|\gamma(z).\]
Then, ${\bf h}_{0}$ belongs to $M_k(SL_2(\ZZ))$.
We have 
\[\langle f_i, \ h_{0}|\gamma \rangle=\langle f_i , \ h_{0} \rangle,\]
for any $\gamma \in SL_2(\ZZ)$, and hence
\[ {\bf L}(l_1,l_2,f_i;\chi_1,\chi_2)\langle f_i, \ f_i \rangle=\langle f_i,{\bf h}_{0} \rangle,\]
and hence we have
\[{\bf h}_{0}(z)=\alpha \widetilde G_k(z)+\sum_{i=1}^d  {\bf L}(l_1,l_2,f_i;\chi_1,\chi_2)f_i(z)\]
with $\alpha \in \CC$.
Put $b_k=\min \{\min_{l_1,l_2} b(k,l_1.l_2), 0\}$ and $a_k=2^{b_k}(k!)^2$,  where $l_1$ and $l_2$ run over all integers satisfying the conditions (D2) and (D3).
By  q expansion principle, for any $\gamma \in SL_2(\ZZ)$, $h_{0}|\gamma$ belongs to 
$M_k(\varGamma^{(1)}(N_1N_2))(\langle a_k^{-1} \rangle_{\frkO_{\QQ(\chi_1,\chi_2,\zeta_N)}[(N_1N_2)^{-1}]})$.Therefore ${\bf h}_{0}$ belongs to 
$M_k(\varGamma^{(1)}(N_1N_2))(\langle a_k^{-1} \rangle_{\frkO_{\QQ(\chi_1,\chi_2,\zeta_N)} [(N_1N_2)^{-1}]}) \cap M_k(SL_2(\ZZ))$. Put $h={\bf h}_0-\alpha \widetilde G_k$.
Then  all the Fourier coefficients of $h$ belong to 
$\langle (2^{b_k}k!^2 \zeta(1-k))^{-1}\rangle_{\frkO_{\QQ(\chi_1,\chi_2,\zeta_N)}[(N_1N_2)^{-1}]}$. 
We note that ${\bf L}(l_1,l_2;f;\chi_1,\chi_2)$ belongs to $\QQ(f,\chi_1,\chi_2)$. Thus, using Lemma \ref{lem.fundamental-lemma}, we can prove the assertion in the same way as Theorem \ref{th.main-result}.
\end{proof}

\begin{corollary} \label{cor.main-result-2}
 Let $f$ be a primitive form in $S_k(SL_2(\ZZ))$. 
Let $\calq_f$ be the set  of prime ideals $\frkp$ of $\QQ(f)$ such that 
\[\ord_{\frkp}(N_{\QQ(f,\chi_1,\chi_2)/\QQ(f)}( {\bf L}(l_1,l_2;f;\chi_1,\chi_2)))
 <0 \]
for some positive integers $l_1,l_2$ and primitive characters $\chi_1,\chi_2$ with $\frkp \nmid m_{\chi_1}, m_{\chi_2}$ 
satisfying the condition $\mathrm{(D1),(D2),(D3)}$.
Then $\calq_f$ is a finite set. Moreover, there exists a positive integer $r$ such that  we have 
\[\ord_{\frkq}({\bf L}(l_1,l_2;f;\chi_1,\chi_2)) \ge -r[\QQ(f,\chi_1,\chi_2):\QQ(f)]\]
for any prime ideal $\frkq$ of $\QQ(f,\chi)$  lying above a prime ideal in $\calq_f$ and  integer $l_1,l_2$ and  primitive  characters $\chi_1,\chi_2$ satisfying the above conditions.   
\end{corollary}

For a prime ideal $\frkp$ of an algebraic number field, let $p=p_{\frkp}$ be a prime number such that $(p_{\frkp})=\ZZ \cap \frkp$. 
Let $K$ a number field containing $\QQ(F)$. Then there exists a semi-simple Galois representation $\rho_f=\rho_{f,\frkp}:\Gal(\bar \QQ/\QQ) \longrightarrow GL_2(K_{\frkp})$ such that 
  $\rho_{f}$ is unramified at a prime number $l \not=p$ and 
\[\det (1_{2}-\rho_{f,\frkp}(\mathrm{Frob}_l^{-1})X)=L_l(X,f),\]
where $\mathrm{Frob}_l$ is the arithmetic Frobenius at $l$, and 
\[L_l(X,f)=1-c_f(l)X+l^{k-1}X^2. \]
 For a $\frkp$-adic representation $\rho$ let $\bar \rho$ denote the mod $\frkp$ representation of $\rho$. To prove our last main result, we provide the following lemma.
\begin{lemma} 
\label{lem.Fontaine-Laffaille-Messing}
Let $p=p_{\frkp}$. 
Let $k$ be a positive even integer such that $k<p$.
Let $f$ be a primitive form in $S_k(SL_2(\ZZ))$.
 Let $a,b$ be  integers such that $-p+1<a< b <p-1$.  Suppose that 
\[\bar \rho_f^{\rm ss} = \bar \chi^a \oplus \bar \chi^b,\]
where $\chi$ is the $p$-cyclotomic character.
Then $(a,b)=(1-k,0)$.
\end{lemma}
\begin{proof}
By [\cite{Buzzard-Gee09}, Theorem 1.2] and  its remark, 
$\overline{\rho}_{f}^{ss}|I_p$ should be 
\[\overline{\chi}^{1-k} \oplus 1\]
 or 
\[{\omega}_2^{1-k} \oplus {\omega}_2^{p(1-k)}\]
 with ${\omega}_2$ the fundamental character of level 2, where $I_p$ denotes the inertia group of $p$ in $\Gal(\bar \QQ/\QQ).$ Thus the assertion  holds.
\end{proof}

Let $f_1,\ldots,f_d$ be a basis of $S_k(SL_2(\ZZ))$ consisting of primitive forms with $f_1=f$ and 
let $\frkD_f$ be the ideal of $\QQ(f)$ generated by all\\
$\prod_{i=2}^d (\lambda_{f_i}(T(m))-\lambda_f(T(m)))$'s ($m \in \ZZ_{>0}$).

\begin{theorem}
\label{th.p-integrality}
Let $f$ be a primitive form in $S_k(SL_2(\ZZ))$. 
Let  $\chi_1$ and $\chi_2$ be primitive characters of conductors $N_1$ and $N_2$, respectively, and let $l_1$ and $l_2$ be positive integers such that   $k-l_1+1 \le l_2 \le l_1-1 \le k-2$.
Let $\frkp$ be a  prime ideal  of $\QQ(f,\chi_1,\chi_2)$ with $p_{\frko} >k$. Suppose that  $\frkp$   divides neither $\frkD_fN_1n_2$ nor $\zeta(1-k)$. Then ${\bf L}(l_1,l_2;f;\chi_1,\chi_2)$ is $\frkp$-integral.
\end{theorem}
\begin{proof}
The assertion follows from Theorem \ref{th.main-result-2}
if $l_1, l_2$ and $\chi_1,\chi_2$ satisfy the conditions (D1),(D2), (D3). Suppose that $l_1=l_2+1$ and $\chi_1$ and $\chi_2$ are trivial. By Lemma \ref{lem.Fontaine-Laffaille-Messing}, there exists a prime number $q_0$ such that $q_0$ is $\frkp$ unit  and 
\[1-c_f(q_0)q_0^{-l_2+1}+q_0^{k-2l_2+1} \not\equiv 0 \text{ mod } \frkp.\]
As stated in the proof of  Theorem \ref{th.main-result-2}, there exists a modular form $g \in M_2(\varGamma_0(q_0))(\frkO_{(\frkp)})$ such that
\[L(s,g)=\zeta(s)\zeta(s-1)(1-q_0^{-s+1}).\]
We can construct a modular form $h_0 \in M_k(\varGamma_0^{(1)}(q_0))$ in the same way as in the proof of  Theorem \ref{th.main-result-2}. 
Then 
\begin{align*} &(1-c_f(q_0)q_0^{-l_2+1}+q_0^{k-2l_2+1}) {\bf L}(l_1,l_2;f_i)  \langle f_i, \ f_i \rangle\\
&=d_0[SL_2(\ZZ):\varGamma^{(1)}_0(q_0)]\langle f_i,h_{0} \rangle \end{align*}
with some integer $d_0$ prime to $\frkp$ 
for any $i=1,\ldots,d$. 
Then by using the same argument as above, we can prove that
\begin{align*}
&\ord_{\frkp}({\bf L}(l_1,l_2;f)(1-c_f(q_0)q_0^{-l_2+1}+q_0^{k-2l_2+1} ))  \ge 0.
\end{align*}
This proves the assertion.
\end{proof}

\end{document}